\documentclass[12pt]{amsart}
\usepackage{amssymb}
\usepackage{amsmath,amscd}
\usepackage{amsfonts,latexsym,verbatim,amscd,mathrsfs,color,array}
\usepackage[all,cmtip]{xy}
\usepackage{mathrsfs}
\usepackage{multirow}
\usepackage{graphicx}
\usepackage{amsfonts, amsthm}
\usepackage{bbm}
\usepackage[hmargin=1in,vmargin=1in]{geometry}
\usepackage{mathdots}
\usepackage{stmaryrd}
\usepackage{MnSymbol}
\usepackage{bbm}
\usepackage[hidelinks]{hyperref}

\allowdisplaybreaks

\def\XXint#1#2#3{{\setbox0=\hbox{$#1{#2#3}{\int}$ }
		\vcenter{\hbox{$#2#3$ }}\kern-.6\wd0}}

\renewcommand{\Re}{\operatorname{Re}}
\renewcommand{\Im}{\operatorname{Im}}
\renewcommand{\P}{\mathbb{P}}
\newcommand{\intprod}{\mathbin{\raisebox{.7\depth}{\scalebox{1}[-1]{$\lnot$}} \,}}
\newcommand{\rintprod}{\mathbin{\raisebox{.7\depth}{\scalebox{-1}[-1]{$\lnot$}}}}
\newcommand{\LB}{\left[}
\newcommand{\RB}{\right]}
\newcommand{\LA}{\langle}
\newcommand{\RA}{\rangle}

\newcommand{\Z}{{\mathbb Z}}

\newcommand{\C}{{\mathbb C}}

\newcommand{\R}{{\mathbb R}}
\newcommand{\G}{{\mathbb G}}

\renewcommand{\H}{{\mathbb H}}
\newcommand{\bi}{{\mathbf i}}
\newcommand{\bj}{{\mathbf j}}
\newcommand{\bk}{{\mathbf k}}

\newcommand{\SU}{{\mathrm{SU} }}
\newcommand{\Sp}{{\mathrm{Sp} }}
\newcommand{\U}{{\mathrm U}}
\newcommand{\SO}{{\mathrm{SO} }}

\renewcommand{\G}{{\mathrm G}}

\newcommand{\Spin}{{\mathrm{Spin} }}
\newcommand{\SL}{{\mathrm{SL} }}
\newcommand{\Sym}{{\mathrm{Sym} }}

\newcommand{\gothg}{{\mathfrak g}}
\newcommand{\so}{\mathfrak {so }}

\newcommand{\scrA}{{\mathscr A}}

\newcommand{\scrG}{{\mathscr G}}

\newcommand{\scrL}{{\mathscr L}}

\newcommand{\scrV}{{\mathscr V}}

\newcommand{\End}{{\text{End}}}

\newcommand{\Lie}{{\mathrm{Lie}}}
\newcommand{\Ric}{{\mathrm{Ric}}}

\newtheorem{thm}{Theorem}[section]
\newtheorem{lemma}[thm]{Lemma}
\newtheorem*{lemma*}{Lemma}
\newtheorem{prop}[thm]{Proposition}

\newtheorem{conj}[thm]{Conjecture}
\newtheorem*{conj*}{Conjecture}

   \newtheoremstyle{others}% name
     {3pt}%      Space above
     {2pt}%      Space below
     {}%         Body font
     {}%         Indent amount (empty = no indent, \parindent = para indent)
     {\bf}% Thm head font
     {.}%        Punctuation after thm head
     {.5em}%     Space after thm head: " " = normal interword space;
     {}%         Thm head spec (can be left empty, meaning `normal')
     
\theoremstyle{others}
\newtheorem{rmk}[thm]{Remark}
\newtheorem*{rmk*}{Remark}
\newtheorem{defn}[thm]{Definition}
\newtheorem{example}[thm]{Example}

\numberwithin{equation}{section}

\begin{document}

\title{$\G_2$-instantons on the $7$-sphere} % via the Hopf fibration}
\author{Alex Waldron}
%\address{\parbox{\linewidth}{Department of Mathematics, University of Wisconsin-Madison \\
%480 Lincoln Dr, Madison, WI 53706}}
\address{Department of Mathematics, University of Wisconsin, Madison, 53706}
\email{waldron@math.wisc.edu}

\begin{abstract}
We study the deformation theory of $\G_2$-instantons on the round 7-sphere, %$S^7,$ 
%obtained from ASD instantons via the quaternionic Hopf fibration.
%certain $\G_2$-instantons on the 7-sphere, %$S^7,$ 
specifically those
obtained from instantons on the 4-sphere
via the quaternionic Hopf fibration.
%specifically,
%those obtained from ASD instantons via the quaternionic Hopf fibration.
%by pullback under the quaternionic Hopf fibration.
%via the quaternionic Hopf fibration. %on a nearly parallel $\G_2$-manifold, focusing on 
%as well as their deformations. %specifically, %, in particular, those obtained by pullback from $S^4$ 
%arising from 
%those obtained by pullback from $S^4$ under the quaternionic Hopf fibration. %, and their deformations. % to $S^4.$ %This produces families of instantons with integral values of the Chern-Simons functional, and we study their deformation theory.
%After taking into account $\Spin(7)$ rotations,
%After pullback to the round $S^7,$
We find that %the ASD instantons on $S^4$ with unit charge and structure group $\SU(2)$
the pullback of the standard ASD instanton %on $S^4$
lies in a smooth, complete, 15-dimensional family of $\G_2$-instantons. In general, the space of infinitesimal $\G_2$-instanton deformations on $S^7$ %of an irreducible instanton on $S^4$
is identified with three copies of the space of ASD deformations on $S^4.$ %, in the case of $S^7_{std},$ or a single copy, in the case of $S^7_{sq}.$ %and 1 copy on the squashed. %finding that its dimension grows quadratically with the charge. %In the case of structure group $SU(2)$ and charge 1, and after taking into account the action of $Spin(7),$ we show that this construction yields a connected component of the moduli space of $G_2$-instantons on $S^7.$
\end{abstract}

\begin{comment}
\begin{abstract}
We study $\G_2$-instantons %on a nearly parallel $\G_2$-manifold, focusing on 
for the two known nearly parallel $\G_2$-structures on $S^7,$ focusing on those %$\G_2$-instantons
 obtained via pullback from $S^4$ under the quaternionic Hopf fibration. %This produces families of instantons with integral values of the Chern-Simons functional, and we study their deformation theory.
%After taking into account $\Spin(7)$ rotations,
The ASD instantons on $S^4$ with unit charge and structure group $\SU(2)$ generate a smooth, complete, 15-dimensional family of $\G_2$-instantons on the standard $S^7.$ In general, we find that the space of infinitesimal deformations of a pullback %from $S^4$
consists of three copies of the space of ASD deformations, in the case of $S^7_{std},$ or a single copy, in the case of $S^7_{sq}.$ %and 1 copy on the squashed. %finding that its dimension grows quadratically with the charge. %In the case of structure group $SU(2)$ and charge 1, and after taking into account the action of $Spin(7),$ we show that this construction yields a connected component of the moduli space of $G_2$-instantons on $S^7.$
\end{abstract}
\end{comment}

\maketitle

\tableofcontents

\thispagestyle{empty}

\section{Introduction}

 %\emph{i.e.}, $\G_2$ and $\Spin(7)$-manifolds. 

%We first recall the basic objects involved.

\subsection{Background}
Let $M$ be an oriented 7-manifold with a $\G_2$-structure, defined by a positive 3-form $\phi.$
Endow $M$ with the corresponding Riemannian metric $g = g_\phi$ and Hodge star operator $* = *_g$ (see \S \ref{ss:npg2} below).

A connection $A$ on a vector bundle $E \to M$ is said to be a \emph{$\G_2$-instanton} if its curvature $F_A$ satisfies
\begin{equation}\label{instantondef1}
F_A + * \left( \phi \wedge F_A \right) = 0.
\end{equation}
This equation appeared in the physics literature during the early 1980s \cite{corrigandevchand}. In the mid-1990s, not long after Joyce's construction \cite{joyceg2} of compact Riemannian manifolds with $\mathrm{Hol}(g) =\G_2,$ Donaldson and Thomas \cite{donaldsonthomas} proposed to define an invariant of (torsion-free) $\G_2$-structures 
by ``counting'' $\G_2$-instantons in a manner analogous to the Casson invariant in dimension three.

Instantons on compact $\G_2$-holonomy manifolds, with structure group $\SU(2)$ or $\SO(3),$ were first constructed by Walpuski \cite{walpuskig2onkummer} in the context of Joyce's Kummer construction \cite{joyceg2}. S{\'a} Earp and Walpuski \cite{saearpwalpuski}, Walpuski \cite{walpuskitcsexample}, and Menet, Nordstr\"om, and S{\'a} Earp \cite{menetnordstromsaearptcs} have succeeded in constructing instantons on several of the twisted connected sum $\G_2$-holonomy manifolds due to Kovalev \cite{kovalevtcs} and Corti-Haskins-Nordstr{\"o}m-Pacini \cite{cortihaskinsnordstrompacini}.
Despite this progress, instantons on compact manifolds with full $\G_2$ holonomy remain extremely difficult to construct. %; for the state of the art, see Walpuski \cite{walpuskiarithmeticinstantons}. %\footnote{See Remark \ref{rmk:instantons} below for further discussion and references.}
The counting program has also encountered a host of fascinating difficulties; %, both at the conceptual and technical levels;
for extensive discussion, see Donaldson-Segal \cite{donaldsonsegal}, Haydys-Walpuski \cite{haydyswalpuski}, Joyce \cite{joyceconjectures16}, and Doan-Walpuski \cite{doanwalpuskicountingassociatives}. %Donaldson-Segal \cite{donaldsonsegal}, Haydys-Walpuski \cite{haydyswalpuski}, Joyce \cite{joyceconjectures16}, and Doan-Walpuski \cite{doanwalpuskicountingassociatives}.

A fruitful alternative is to consider metrics with less-than-full holonomy or $\G_2$-structures with certain types of torsion. The instantons that appear in these contexts can demonstrate geometric phenomena that we ultimately hope to observe in the torsion-free case.

One such class is the \emph{nearly parallel} $\G_2$-structures, which satisfy
\begin{equation}\label{nearlypar}
d \phi = \tau_0 \psi.
\end{equation}
%where $\psi = *_{g_{\phi}} \phi$ is the dual 4-form.
Here $\psi = *_{g_{\phi}} \phi$ is the dual 4-form and $\tau_0$ is a %nonzero
constant.
%Among coclosed $\G_2$-structures, the 
%Compared with general co-closed $\G_2$-structures, 
Nearly parallel $\G_2$-structures
have much in common with the torsion-free case: their associated metrics are Einstein and they enjoy a natural spinorial description \cite{friedrichkathmoroianusemmelmann}. More importantly for this work, instantons on nearly parallel $\G_2$-manifolds are critical points of the Yang-Mills functional (see \cite{harlandnolle} or (\ref{instareym}) below). This stands in contrast with general co-closed $\G_2$-structures. %is far from true of general co-closed $\G_2$-structures. %They are also somewhat easier to come by
%This makes gauge theory on nearly parallel $\G_2$-manifolds an attractive area of study.
%Here, we give $S^7$ its standard nearly parallel $\G_2$ structure induced from the octonionic structure of $\R^8.$
%In addition to the desire to better understand gauge theory on nearly parallel $\G_2$-manifolds,
%The reasons for studying the special case of the standard $S^7$ are even more basic, however. %$\G_2$-instantons specifically on the standard 7-sphere.

%Perhaps the strongest appeal of nearly parallel $\G_2$ structures is derived from their most famous example.
Most famously, the 7-sphere carries a ``standard'' $\G_2$-structure induced by the octonionic structure of $\R^8;$
for elementary reasons, this turns out to be nearly parallel. (In fact, $S^7$ also carries a second, ``squashed,'' nearly parallel $\G_2$-structure---see Example \ref{ex:squashed} below and Appendix \ref{appendix:squashed}---but we shall be concerned primarily with the standard structure.)
There are three main reasons why the standard $S^7$ is an especially appealing arena for gauge theory.

First, determining the complete moduli space of instantons on $S^7$
%The first motivation is
presents a clear challenge for higher-dimensional gauge theory. This problem is at least as difficult as the corresponding problem on $S^4,$ solved by the spectacular theorem of Atiyah-Drinfeld-Hitchin-Manin \cite{adhm}. At present, we are limited to studying the deformations of a given $\G_2$-instanton on $S^7,$ in the spirit of Atiyah-Hitchin-Singer's classic work on deformations of ASD instantons on $S^4$
\cite{ahs}. %As we shall see, instantons on $S^7$ tend to have large spaces of infinitesimal deformations---one of many ``pathological phenomena'' that may also potentially occur in the torsion-free setting.

The second motivation comes from another well-known difficulty in higher-dimensional gauge theory: the appearance of %$\G_2$- and $\Spin(7)$-
instantons with essential singularities. Given a $\Spin(7)$-instanton with an isolated singularity on a $\Spin(7)$-manifold,
the cross-section of the tangent cone %at the singular point
is a $\G_2$-instanton on the standard $S^7.$ %$\left( S^7, \phi_{std} \right).$
Hence, these are the ``building blocks'' for the simplest non-removable singularities in dimension eight. %for instantons on exceptional-holonomy manifolds. %In order to study the deformation theory of a singular instanton on a $\Spin(7)$-manifold, it is necessary to understand the deformations of the tangent cone at a singular point. %, which should be the simplest examples of essential singularities in gauge theory.

The third motivation is the relationship with gauge theory in fewer than seven dimensions. %Since there is at present no ``direct'' method to construct instantons on a given compact manifold with full $\G_2$ holonomy, %one always works 
Since $\G_2$-holonomy metrics typically collapse at the boundary of the moduli space, %---for instance, by forming orbifold points, undoing a connect sum, or collapsing---
it is important to understand instantons on model $\G_2$-manifolds coming from lower-dimensional geometries. Recently, Y. Wang \cite{yuanqicyproduct} has shown that %for $M = S^1 \times X,$ with 
any $\G_2$-instanton on $S^1 \times X,$ for $X$ a Calabi-Yau 3-fold, is equivalent by a broken gauge transformation to the pullback of a Hermitian-Yang-Mills connection. % on $X.$
This is the first case where the full moduli space of $\G_2$-instantons on a given 7-manifold with holonomy contained in $\G_2$ has been identified.

In the present case, a convenient link with 4-dimensional gauge theory
is provided by the \emph{quaternionic Hopf fibration}
\begin{equation}\label{quathopf}
S^3 \to S^7 \to S^4.
\end{equation}
With the proper conventions, the pullback %under (\ref{quathopf})
of any anti-self-dual (ASD) instanton on $S^4$
is a $\G_2$-instanton on $S^7.$ This provides an abundant source of examples. The resulting families of $\G_2$-instantons can be further enlarged by the action of the automorphism group $\Spin(7),$ %of $\phi_{std},$
which interacts with the Hopf fibration in a nontrivial way.
Our main results %, Theorems \ref{thm:grassmann} and \ref{thm:generaldefos},
concern the dimension and completeness properties of these families.

\subsection{Summary}

In \S \ref{sec:conventions}-\ref{sec:instantons}, we set our conventions and derive the basic results concerning instantons on nearly parallel $\G_2$-manifolds, while reviewing the literature in this area.
%Most of the results in \S 2 are known to researchers informally or by analogy with the nearly K\"ahler case (see Xu \cite{xunearlykahler}), but have not appeared in the literature in their present form. For the spinorial formulation, see Harland-N\"olle \cite{harlandnolle}.

In \S \ref{sec:quatinst}, we take a concrete approach to the special case of the standard instanton. In quaternionic notation on $\R^8 \cong \H^2,$
the pullback by the (right) Hopf fibration of the standard ASD instanton %on $S^4$
may be given by %the connection form
\begin{equation}\label{standardg2instanton}
A_0(x,y) = \frac{1}{|x|^2 + |y|^2} \Im \LB |y|^2  x^{-1} \, dx + |x|^2 y^{-1} \, dy \RB.
\end{equation}
\begin{comment}
As written here, this defines a smooth connection over %$\Sp(2)$
 $\R^8\setminus \{0\}$ %with structure group $\SU(2),$
with curvature %2-form
valued in $\Lie(\Sp(2)) \otimes \Im \H;$ in particular, it is a $\Spin(7)$-instanton. % with structure group $\SU(2).$ %\subset \Lambda^2 \R^8.$
\end{comment}
The restriction of $A_0$ to $S^7$ defines a smooth connection, dubbed the \emph{standard $\G_2$-instanton}. Proposition \ref{prop:15deformations} describes
a 15-dimensional space of infinitesimal deformations of $A_0,$ generated by the deformations coming from $S^4$ together with the $\Spin(7)$ rotations of $S^7,$ modulo gauge. %(\ref{15deformations})
 %generated by the deformations on $S^4$ combined with the action of $\Spin(7),$ modulo gauge equivalence. This
This will turn out to be the full space of deformations of $A_0$ as a $\G_2$-instanton. % of $A_0$ as a $\G_2$-instanton.

In \S \ref{sec:generaldefos}, we prove our main result, Theorem \ref{thm:generaldefos}, which determines the space of infinitesimal deformations of the pullback to $S^7$ %by the Hopf fibration %(\ref{quathopf}) %to $S^7$
of an arbitrary irreducible ASD instanton on $S^4.$
The argument is elementary, but far from trivial. The difficulty is caused by two factors: first, according to the theorem of Bourguignon-Lawson-Simons \cite{bourguignonlawson,bourguignonlawsonsimons}, the Yang-Mills stability operator necessarily has negative eigenvalues. Second, the $\Sp(1)$-action giving the quaternionic Hopf fibration %is not a subgroup of $\Spin(7),$ %,
does not commute with the deformation operator (see Remark \ref{rmk:convention} below).

%#v3

To obtain the result, we first prove a vanishing theorem for the vertical component of an infinitesimal deformation in Coulomb gauge, Theorem \ref{thm:vanishing}. This follows from a delicate analysis of the Weitzenbock formula for the %square of the deformation operator
stability operator, in which the first-order deformation equation is used crucially. %Fortunately, these can be avoided based on the first-order deformation equation for a $\G_2$-instanton. %Fortunately, the vertical part of the Laplacian exactly compensates, giving a nonnegative operator whose kernel vanishes in the irreducible case. %In the end one shows that the vertical part of the deformation operator is indeed nonnegative, with kernel coming from the stabilizer of the connection downstairs; this vanishes by assumption.
Having established that the vertical component vanishes, we use another squaring trick, Lemma \ref{lemma:horiz}, to calculate the horizontal component of an infinitesimal deformation. This leads directly to Theorem \ref{thm:generaldefos}, which equates the full space of infinitesimal deformations of the pullback, %to $S^7,$
as a $\G_2$-instanton, %of a pullback to $S^7_{std}.$ %of an irreducible ASD instanton on $S^4.$
%This consists of
with three copies of the ASD deformations on $S^4.$ 

In \S \ref{sec:global}, we briefly discuss %the implications of Proposition \ref{prop:15deformations} and Theorem \ref{thm:generaldefos} regarding
the global structure of these families of $\G_2$-instantons. %se families of instantons.
In the case of charge $\kappa = 1$ and structure group $\SU(2),$ which includes the standard instanton, %$\G_2$-instanton (\ref{standardg2instanton}), 
%Proposition \ref{prop:15deformations} and Theorem \ref{thm:generaldefos} give
we have the following result. % the following description.

\begin{thm}\label{thm:grassmann} The connected component of $A_0$ in the moduli space of $\G_2$-instantons on $S^7$ is diffeomorphic to the tautological 5-plane bundle over the oriented real Grassmannian $\G^{or}(5,7).$
\end{thm}

\noindent 
In this description, the base space %$\G^{or}(5,7)$ 
corresponds to the orbit of the Hopf fibration (\ref{quathopf}) under conjugation by $\Spin(7),$ and the fiber corresponds to the pullback of the unit-charge ASD moduli space. % on $S^4.$
The total space is 15-dimensional, agreeing with %Proposition \ref{prop:15deformations} and
the dimension formula of Theorem \ref{thm:generaldefos}.
%The case of higher charge is more complicated, and
By contrast, in the case of higher charge,
we do not expect all of the infinitesimal deformations identified by Theorem \ref{thm:generaldefos}
to be integrable (see \S \ref{ss:discussion} below). %#v3
%as far as obtaining an analogue of the ADHM theorem for $S^7$ in the future, 

Lastly, we state Conjecture \ref{conj:donaldson}, due to Donaldson, which asserts that every $\G_2$-instanton on $S^7$ having integral Chern-Simons value should arise from the pullback construction.

\subsection{Acknowledgements} The author thanks Simon Donaldson %for suggesting the problem, and thanks Prof. Donaldson 
and Thomas Walpuski for discussions, and thanks Aleksander Doan for comments on the manuscript. He also thanks two anonymous referees for helpful suggestions. This research was supported during 2017-18 by the Simons Collaboration on Special Holonomy.

\section{Conventions}\label{sec:conventions} % for $\G_2$ and $\Spin(7)$}

%#v3

%In this section, we 
%Most of the results of this section are more-or-less known by analogy with the nearly K\"ahler case, and have appeared in an equivalent form in Harland and Nolle.

\begin{comment}
Let $\R^8 = \R^4 \oplus \R^{4'},$ and choose bases $\omega_i, \omega'_i,$ $i = 1, 2,3,$ for the self-dual 2-forms on $\R^4$ and $\R^{4'},$ respectively. We may take
\begin{equation}
\begin{split}
\omega_1 = dx^{12} + dx^{34}, \qquad \omega'_1 = dx^{56} + dx^{78} \\
\omega_2 = dx^{13} - dx^{24}, \qquad \omega'_2 = dx^{57} - dx^{68} \\
\omega_3 = dx^{14} + dx^{23}, \qquad \omega'_3 = dx^{58} + dx^{67}. \\
\end{split}
\end{equation}
\end{comment}

\subsection{Quaternions, 2-forms, and Fueter maps}\label{ss:quatsandfueter} Let $\H$ denote the 4-dimensional algebra of quaternions, generated by $\mathbf{1}, \bi, \bj,\bk,$ and subject to the relations
\begin{equation*}
\bi^2 = \bj^2 = \bk^2 = \bi \bj \bk = -\mathbf{1}.
\end{equation*}
The 3-dimensional space of imaginary quaternions $\Im \H,$ with commutator bracket, %$\LB , \RB,$
is isomorphic to the Lie algebra $\mathfrak{su}(2).$

\begin{comment}
Then the 2-forms (\ref{2formbases}) act on $\R^4,$ in the basis (\ref{quaternionbasis}), by quaternion multiplication on the right: %the anti-self-dual 2-forms
\begin{equation}\label{quaternionaction}
\cdot \left( - \bi \right) =  \intprod \mathbf{\omega}_1, \qquad
\cdot \left( - \bj \right) = \intprod \mathbf{\omega}_2, \qquad
\cdot \left( - \bk \right)  = \intprod \mathbf{\omega}_3 .
\end{equation}
%Under right-multiplication, we have $\left( - \cdot \bi \right) = \omega_1,$ etc., for the basis given by (\ref{omegaidef}).
We define the three complex structures $I_1, I_2, I_3$ on $\R^4.$ 
Then, in the basis (\ref{quaternionbasis}), one can check that
\begin{equation}\label{quaternionaction}
I_1 = \cdot \left( - \bi \right) , \qquad
I_2 = \cdot \left( - \bj \right) , \qquad
I_3 = \cdot \left( - \bk \right) .
\end{equation}
\end{comment}

We shall identify $\H$ with $\R^4$
%$$\H = \{\bbi, \bbj, \bbk \mid \bbi^2 = \bbj^2 = \bbk^2 = \bbi \bbj \bbk = -1 \} $$
as follows:
\begin{equation}\label{quaternionbasis}
\{ e_0, e_1, e_2, e_3 \} = \{ \mathbf{1}, -\bi, -\bj, -\bk\}.
\end{equation}
Denote by $\H_l$ and $\H_r$ the commuting subalgebras of $\mathrm{End} \, \R^4$ corresponding to left- and right-multiplication by $\H,$ respectively, and let $\Sp(1)_l$ and $\Sp(1)_r$ be the corresponding subgroups of $\SO(4)$ generated by unit quaternions.
With this convention, we have
\begin{equation*}
\Lie(\Sp(1)_l) = \Lambda^{2-}, \qquad \Lie(\Sp(1)_r) = \Lambda^{2+}
\end{equation*}
where $\Lambda^{2 \pm} \subset \Lambda^2 \R^4$ denotes the space of (anti-)self-dual 2-forms with respect to the Euclidean metric.

Choose the following standard basis for the self-dual 2-forms on $\R^4:$ 
\begin{equation}\label{2formbases}
 \mathbf{\omega}_1 = dx^{01} + dx^{23}, \qquad \mathbf{\omega}_2 = dx^{02} - dx^{13}, \qquad \mathbf{\omega}_3 = dx^{03} + dx^{12}.
\end{equation}
Here we abbreviate $dx^{01} = dx^0 \wedge dx^1,$ {\it etc.} For $i = 1,2,3,$ define the complex structure $I_i$ on $\R^4$ %$i = 1,2,3,$
%$I_1, I_2,$ and $I_3,$
by %(given by the 2-forms $-\omega_1, -\omega_2,-\omega_3$ above, respectively). %#v3
\begin{equation}\label{quaternionaction}
\LA I_i(v), w \RA = \mathbf{\omega}_i(v,w) \quad \forall \,v,w \in T \R^4. %\qquad \mbox{ for } i = 1,2,3,
%\LA I_2(v), w \RA = \mathbf{\omega}_2(v,w), \qquad
%\LA I_3(v), w \RA  = \mathbf{\omega}_3(v,w),
\end{equation}
%for all $v,w \in T \R^4.$ %where we identify tangent and cotangent vectors using the Euclidean metric. %
Under the identification (\ref{quaternionbasis}), $I_i$ corresponds to right-multiplication by the element $e_i.$ %$I_1, I_2,$ and $I_3$ correspond, respectively, to right-multiplication by $e_1, e_2,$ and $e_3,$ in the basis (\ref{quaternionbasis}).  %With the convention (\ref{quaternionbasis}), the Lie algebras $\mathfrak{sp}(1)_l$ and $\mathfrak{sp}(1)_r$ correspond, respectively, to anti-self-dual and self-dual 2-forms in $\Lambda^2\R^4 \cong \mathfrak{so}(4).$
%With the convention (\ref{quaternionbasis}), the Lie algebras $\mathfrak{sp}(1)_l$ and $\mathfrak{sp}(1)_r$ correspond, respectively, to anti-self-dual and self-dual 2-forms in $\Lambda^2\R^4 \cong \mathfrak{so}(4).$

%and satisfy $I_1 \circ I_2 = I_3.$
%Under the splitting
%$$ \Lambda^2 \R^4 = \Lambda^{2-} \oplus \Lambda^{2+}$$
%Then $I_1$ is the standard complex structure, and we have $I_1 \circ I_2 = I_3.$

We define a \textit{Fueter map} $L : \R^4 \to \R^4$ to be an endomorphism satisfying %\footnote{See Remark \ref{rmk:convention} below for an about this convention.}
\begin{equation}\label{fueter}
L + I_1 L I_1 + I_2 L I_2 - I_3 L I_3 = 0.
\end{equation}
The 12-dimensional subspace of Fueter maps $\mathfrak{F} \subset \End \, \R^4$ is the direct sum:
\begin{equation}\label{fueterfactors}
\mathfrak{F} = \H_l \oplus \H_l I_1 \oplus \H_l I_2. %\subset \mathrm{End} \, \R^4.
\end{equation}
%Notice that (\ref{fueterfactors}) contains the space %$\H_l \oplus \H_l I_1$
%of all
In particular, $\mathfrak{F}$ contains the space of linear maps for the standard complex structure, $I_1.$ %any complex-linear map for the standard complex structure, $I_1,$ is a Fueter map. % is a subspace of (\ref{fueterfactors}).
%which is just the orthogonal complement of $\H_l$ in $\mathrm{End} \, \R^4.$ 
%Letting $\H_x$ and $\H_y$ be two identical copies of the quaternions,
%Writing
%\begin{equation*}
%\R^8 = \H \oplus \H
%\end{equation*}

\subsection{$\Spin(7)$ and its subgroups}\label{ss:spin7andg2} %We shall write $$\R^8 = \R^4_x \oplus \R^4_y.$$
%Let $\omega_i^{x}$ and $\omega_i^y$ denote the bases of self-dual 2-forms defined by (\ref{2formbases}) on $\R^4_x$ and $\R^4_y,$ respectively.
The standard 4-form on $\R^8 = \R^4_x \oplus \R^4_y$ is defined by
\begin{equation}\label{psidef}
\Psi_0 = dx^{0123} + dy^{0123} + \omega^x_1 \wedge \omega^y_1 + \omega^x_2 \wedge \omega^y_2 - \omega^x_3 \wedge \omega^y_3.
\end{equation}
%where $\omega_i^{x,y}$ are given by (\ref{2formbases}).
\begin{comment}
\begin{equation}\label{omegaidef}
\begin{split}
\omega^x_1 & = dx^{01} + dx^{23}, \qquad \omega^y_1 = dy^{01} + dy^{23} \\
\omega^x_2 & = dx^{02} - dx^{13}, \qquad \omega^y_2 = dy^{02} - dy^{13} \\
\omega^x_3 & = dx^{03} + dx^{12}, \qquad \omega^y_3 = dy^{03} + dy^{12}.
\end{split}
\end{equation}
\end{comment}
%are bases for the self-dual 2-forms on $\R^4_x$ and $\R^4_y,$ respectively.
The group $\Spin(7)$ consists of all linear transformations of $\R^8$ that preserve $\Psi_0$ under pullback. It is a simply-connected, simple, Lie subgroup of $\SO(8)$ of dimension 21 (see {\it e.g.} Walpuski and Salamon \cite[\S 9]{walpuskimasters}).

The Lie algebra of $\Spin(7)$ corresponds to the subspace of 2-forms $\eta \in \Lambda^2 \R^8 \cong \so(8)$ %#v3 
satisfying
\begin{equation}\label{spin7def}
\eta + \ast \left( \Psi_0 \wedge \eta \right) = 0.
\end{equation}
%This can be described explicitly as follows.
%Note that a Fueter map $L = \left( L_{ij} \right),$ According to(\ref{fueter}), naturally defines a 2-form on $\R^8$ by the prescription
%\begin{equation*}
%L_{ij} dy^i \wedge dx^j.
%\end{equation*}
Let $\Lambda^{2+}_{d} \subset \Lambda^2 \R^8$ be the subalgebra spanned by the three elements %\footnote{Notice the sign difference on the third element.} % of (\ref{lambdad}).}
\begin{equation}\label{omegad123}
\omega_1^x - \omega_1^y, \qquad \omega_2^x - \omega_2^y, \qquad \omega_3^x + \omega_3^y.
\end{equation}
\begin{comment}
\begin{equation*}%\label{lambdad}
 \left( \begin{array}{cc}
I_1 &  0 \\
0 & - I_1 \end{array} \right), \quad \left( \begin{array}{cc}
I_2 &  0 \\
0 & - I_2 \end{array} \right), \quad \left( \begin{array}{cc}
I_3 &  0 \\
0 & I_3 \end{array} \right).
\end{equation*}
\end{comment}
Also denote the subspace
\begin{equation}\label{frakFdef}
\mathfrak{F}_{x,y} = \left\{L_{ij} dy^i \wedge dx^j \mid L \in \mathfrak{F} \right\} \subset \Lambda^2 \R^8.
\end{equation}
Then we have the following decomposition:
\begin{equation}\label{spin7splitting}
\Lie(\Spin(7)) = \Lambda^{2-}_{x} \oplus \Lambda^{2-}_{y} \oplus \Lambda^{2+}_{d} \oplus \mathfrak{F}_{x,y} . %\{\text{Fueter maps } \R^4_x \to \R^4_y \}.
\end{equation}
One readily checks that each factor of (\ref{spin7splitting}) satisfies (\ref{spin7def}).
%The isomorphism $\mathfrak{spin}(7) \tilde{\to} \, \mathfrak{so}(7)$ is given by the adjoint action on the 7-dimensional orthogonal complement of $\mathfrak{spin}(7)$ in $\mathfrak{so}(8).$
\begin{comment}
%(quaternion multiplication by $\bi$ on the right)
The standard complex structure on $\R^8 = \C^4 = \H^2$ is given by
\begin{equation*}%\label{complexstructure}
\left( \begin{array}{cc}
I_1 &  0 \\
0 & I_1 \end{array} \right).
\end{equation*}
\end{comment}

Take complex coordinates
$$z^1 = x^0 + i x^1, \quad z^2 = x^2 + i x^3, \quad z^3 = y^0 + i y^1, \quad z^4 = y^2 + i y^3$$
for $\R^8 \cong \C^4.$ %with the standard complex structure $I_1.$
We have the standard holomorphic volume form and K\"ahler form
\begin{equation}\label{stdkahlerforms}
\Omega = dz^1 \wedge dz^2 \wedge dz^3 \wedge dz^4, \qquad {\bf \omega} = \frac{i}{2} \sum_{i = 1}^4 dz^i \wedge d\bar{z}^i = \omega_1^x + \omega_1^y.
\end{equation}
One may verify that
\begin{equation}\label{psi0kahler}
\Psi_0 =  \frac{1}{2} \omega \wedge \omega + \Re \Omega.
\end{equation} %Done 9/8/19 w/opposite sign.
The 15-dimensional group $\SU(4)$ %with its standard presentation, #v3 
is therefore a subgroup of $\Spin(7).$ In particular, the 10-dimensional group $\Sp(2)$ of orthogonal quaternionic matrices, linear over $I_1, I_2,$ and $I_3,$ is also a subgroup.
%\begin{rmk} Notice that in our convention, while the group $\Sp(1)_{r,diag}$ is contained in $\Spin(7),$ the group $\Sp(2)_l$ of $\H_r$-linear isometries of $\R^8 = \H^2,$ is \emph{not}. % a subgroup of $\Spin(7).$
%This is necessarily the case, because although $\Spin(7)$ contains many copies of $\Sp(1)$ and $\Sp(2),$ it contains no Lie subgroup isomorphic to $\Sp(2) \times \Sp(1).$ %***check this.
 %while $\Sp(1)_r.$ %the group of unit quaternions acting on the right, is.
%\end{rmk}
%\subsection{The Lie group $\G_2$}

The 14-dimensional Lie group $\G_2$ is the subgroup of $\Spin(7)$ stabilizing a point on $S^7.$ Equivalently, $\G_2$ is the subgroup of $\mathrm{GL}(7)$ that preserves the model 3-form
\begin{equation}\label{phi0}
\phi_0 = \frac{\partial}{\partial x^0} \intprod \Psi_0 = dx^{123} + dx^1 \wedge \omega_1^y + dx^2 \wedge \omega_2^y - dx^3 \wedge \omega_3^y.
\end{equation}
Denote the dual 4-form by
\begin{equation}\label{psi0}
\psi_0 = \ast_{\R^7} \phi_0 = dy^{1234} + dx^{23} \wedge \omega_1^y - dx^{13} \wedge \omega_2^y - dx^{12} \wedge \omega_3^y.
\end{equation}
The Lie algebra $\Lie(\G_2)$ corresponds to the subspace of 2-forms $\xi \in \Lambda^2 \R^7$ satisfying
\begin{equation*}
\xi + * \left( \phi_0 \wedge \xi \right) = 0,
\end{equation*}
or equivalently
\begin{equation*}
\psi_0 \wedge \xi = 0.
\end{equation*}

\begin{rmk}\label{rmk:convention}
The fact that $\Sp(2)\Sp(1)$ is not a subgroup of $\Spin(7)$ causes an essential difficulty. With our convention, $\Sp(2)$ is a subroup, but the block-diagonal $\Sp(1),$ giving the Hopf fibration, is not.
On the other hand, with the convention used for instance by Walpuski \cite{walpuskispin7instantons}, $\Sp(1)$ is a subgroup, but the commuting $\Sp(2)$ is not. Our convention is necessary for Lemma \ref{lemma:instapullback} below. %, because $\Sp(2)$ is not a subgroup.
%is necessary for  below. %Since $\Sp(2) \Sp(1) \not\subset \Spin(7),$
%In fact, there is no convention in which both $\Sp(2)$ (giving a transitive group of automorphisms) and a commuting copy of $\Sp(1)$ (giving the Hopf fibration) are subgroups of $\Spin(7);$
%this is an essential difficulty.
%This adds a substantial difficulty. % This makes the problem substantially more difficult. %adds considerably to the difficulty of the problem.
%makes the problem much more difficult.
% the difficulty of the paper significantly.
%makes the problem at hand significantly more difficult. %contributes substantially to the difficulty of the paper. %This is the essential difficulty of the paper.
%#v3

%contains no subgroup isomorphic to  We have chosen the present convention in view of Lemma \ref{lemma:instapullback} below. %, where $\Sp(2)_l \subset SU(4) \subset \Spin(7),$ % is chosen such that the pullback of an asd 2-form on $S^4$ lies inside $spin(7)$ 
 %the convention in which $\Sp(2)_l$ is a subgroup %the convention in which $\Sp(2)_l$ is a subgroup
%, which fails in the other (standard) convention.}

\end{rmk}

\section{Instantons on nearly parallel $\G_2$-manifolds}\label{sec:instantons}

In this section, we establish the basic facts about instantons on nearly parallel $\G_2$-manifolds. Most of the results are known to researchers informally or by analogy with the nearly K\"ahler case (see Xu \cite{xunearlykahler}), but some (in particular Proposition \ref{prop:weitz}) have not appeared %in the literature
in their present form. For the spinorial formulation, see Harland-N\"olle \cite{harlandnolle}.

\subsection{Nearly parallel $\G_2$-structures}\label{ss:npg2} Let $M$ be an oriented 7-manifold. Recall that a \emph{$\G_2$-structure} on $M$ is defined by a global 3-form $\phi$ that is \emph{positive}, in the sense that
\begin{equation}
G_\phi(v) = \left( v \intprod \phi \right) \wedge \left( v \intprod \phi \right) \wedge \phi > 0
\end{equation}
for all $x \in M$ and $v \neq 0 \in T_x M.$ Any such $\phi$ is pointwise equivalent to the model 3-form $\phi_0,$ given by (\ref{phi0}) above (see \cite[Theorem 3.2]{walpuskimasters}).

A positive 3-form defines a unique Riemannian metric $g_\phi$ on $M$ by the requirement
\begin{equation}\label{g2metric}
6 g_\phi(v,v)  Vol_{g_\phi} = G_\phi(v) \,\, \forall \,\, v \in TM. 
\end{equation}
We also associate to $\phi$ the dual 4-form
\begin{equation}
\psi = \ast_{g_{\phi}} \phi.
\end{equation}
Recall that a $\G_2$-structure is said to be \emph{closed} if $d \phi = 0,$ and \emph{coclosed} if $d \psi = 0.$

We are concerned with $\G_2$-structures satisfying (\ref{nearlypar}), where we assume
%\begin{equation}\label{nearlypar}
%d \phi = \tau_0 \psi
%\end{equation}
%for a constant $\tau_0 \neq 0.$ always assume
\begin{equation}\label{pm4}
\tau_0 = \pm 4.
\end{equation}
The basic reference for nearly parallel structures is Friedrich-Kath-Moroianu-Semmelmann \cite{friedrichkathmoroianusemmelmann}. With the normalization (\ref{pm4}), nearly parallel $\G_2$-manifolds are Einstein, with 
\begin{equation}
\mathrm{Ric}_g = 6 g.
\end{equation}
%\begin{rmk}
There are three further equivalent formulations of the nearly parallel condition (\ref{nearlypar}).
The first is that the induced $\Spin(7)$-structure on the cone over $M$ be torsion-free. The second (and most frequently used) condition %condition states that a nearly parallel $\G_2$-manifold is one
is that $M$ possess a nonzero Killing spinor. %\cite{friedrichkathmoroianusemmelmann}.
The third is as follows: %that the full covariant derivative of $\phi$ be proportional to $\psi$: %3-form $\phi$ is not torsion-free, but the following Lemma shows that its torsion is particularly simple.
\begin{lemma}\label{lemma:gradphi} A $\G_2$-structure $\phi$ is nearly parallel if and only if
\begin{equation*}
\nabla \phi = \frac{\tau_0}{4} \psi.
\end{equation*}
Here $\nabla$ is the Levi-Civita connection associated to the metric $g_\phi$ defined by (\ref{g2metric}).
\end{lemma}
\begin{proof} See Karigiannis \cite[Theorem 2.27]{karigiannisflowsofg2}.
\end{proof}
%\end{rmk}

\begin{example} Define the \emph{standard $\G_2$-structure on $S^7$} by
\begin{equation}\label{phistdzerothdef}
\phi_{std} = \left. \vec{r} \intprod \Psi_0 \right|_{S^7},
\end{equation}
where
$$\vec{r} = x^i \frac{\partial}{\partial x^i} + y^i \frac{\partial}{\partial y^i}$$
is the coordinate vector field on $\R^8.$
The group of global automorphisms of $\phi_{std}$ is $\Spin(7).$

To obtain a more explicit expression for $\phi_{std},$ we define the 3-form
\begin{equation*}
\nu^x = \vec{r} \intprod \mathrm{Vol}_{\R^4_x}= x^0 dx^{123} - x^1 dx^{023} + x^2 dx^{013} - x^3 dx^{012}
\end{equation*}
on $\R^4_x,$ and the 1-forms
\begin{equation}\label{zetaidef}
\begin{split}
\zeta_1^x &= \vec{r}\intprod \omega_1^x = x^0 dx^1 - x^1 dx^0 + x^2 dx^3 - x^3 dx^2 \\
\zeta_2^x &= \vec{r} \intprod \omega_2^x = x^0 dx^2 - x^2 dx^0 - x^1 dx^3 + x^3 dx^1 \\
\zeta_3^x & = \vec{r} \intprod \omega_3^x = x^0 dx^3 - x^3 dx^0 + x^1 dx^2 - x^2 dx^1.
\end{split}
\end{equation}
Define $\nu^y$ and $\zeta_i^y$ similarly. We then have
\begin{equation}\label{phistdfirstdef}
\begin{split}
\phi_{std} 
& = \nu^x + \nu^y + \zeta_1^x \wedge \omega_1^y + \zeta_1^y \wedge \omega_1^x + \zeta_2^x \wedge \omega_2^y + \zeta_2^y \wedge \omega_2^x - \zeta_3^x \wedge \omega_3^y - \zeta_3^y \wedge \omega_3^x.
\end{split}
\end{equation}
It is easy to check, using the $\Spin(7)$-invariance, that $\phi_{std}$ defines the round metric %$g_{std}$
on $S^7.$

Notice that $d \nu^x = 4 {\mathrm{Vol}}_{\R^4_x}$ and $d \zeta_i^x = 2 \omega_i^x,$ and similarly for $y,$ so %{\it etc.}, and so
\begin{equation}\label{phiPsi0}
d \phi_{std} = \left. 4 \Psi_0 \right|_{S^7}.
\end{equation}
Also note that
\begin{equation}\label{psistd}
\psi_{std} = \ast_{std} \phi_{std} = \left. \ast_{\R^8} \Psi_0 \right|_{S^7} = \left. \Psi_0 \right|_{S^7}.
\end{equation}
%where the last equality follows from self-duality of $\Psi_0$ as a 4-form on $\R^8.$
It follows from (\ref{phiPsi0}-\ref{psistd}) that $\phi_{std}$ is a nearly parallel $\G_2$-structure.

\begin{comment}
It will be useful to rewrite $\phi_{std}$ using global $\Sp(2)$-invariant coframes. Let
%A global $\Sp(2)_l$-invariant frame for the $\Omega^1_v$ is given by
\begin{equation*}%\label{gammaidef}
\begin{split}
\zeta_i & = \zeta_i^x + \zeta_i^y, \qquad \omega_i^\circ = \left. \omega^x_i + \omega^y_i \right|_{S^7} \\
\bar{\omega}_i  & = \omega_i^\circ - \frac{1}{2}\epsilon_{ijk} \zeta_j \wedge \zeta_k, \qquad i = 1,2,3 \\
\nu & = \zeta_1 \wedge \zeta_2 \wedge \zeta_3.
\end{split}
\end{equation*}
%where $\zeta_i^x, \zeta_i^y$ %and $\omega_i^x, \omega_i^y$ 
%are defined by (\ref{gammaidef}), and %and (\ref{2formbases}), respectively. Let
From (\ref{phistdfirstdef}) and the $\Sp(2)$-invariance, we have
\begin{equation}%\label{phistdseconddef}
\begin{split}
\phi_{std} 
& = \nu + \zeta_1 \wedge \bar{\omega}_1 - \zeta_2 \wedge \bar{\omega}_2 + \zeta_3 \wedge \bar{\omega}_3.
\end{split}
\end{equation}
\end{comment}
\end{example}
\begin{example}\label{ex:squashed} Define the \emph{squashed $\G_2$-structure}
\begin{equation*}
\phi_{sq} = \frac{27}{25} \left( \begin{split} \frac{1}{5} (\nu_x + \nu_y) + \frac{16}{5} \left( \begin{split}\zeta_1^x \wedge \zeta_2^x \wedge \zeta_3^y + \zeta_1^x \wedge \zeta_2^y \wedge \zeta_3^x + \zeta_1^y \wedge \zeta_2^x \wedge \zeta_3^x \\ + \zeta_1^x \wedge \zeta_2^y \wedge \zeta_3^y + \zeta_1^y \wedge \zeta_2^x \wedge \zeta_3^y + \zeta_1^y \wedge \zeta_2^y \wedge \zeta_3^x  \end{split} \right) \\
- \zeta^x_1 \wedge \omega^y_1 - \zeta_1^y \wedge \omega_1^x - \zeta^x_2 \wedge \omega^y_2 - \zeta_2^y \wedge \omega_2^x - \zeta^x_3 \wedge \omega^y_3 - \zeta_3^y \wedge \omega_3^x \end{split}\right).
\end{equation*}
%The corresponding metric $g_{sq}$ shrinks the orbits of $\Sp(1)_r$ by a factor of $1/ \sqrt{5}$ from the standard metric.
%The stabilizer of $\phi_{sq}$ is $\Sp(2) \Sp(1)$ (Kawai \cite{kawai}, Prop.).
The squashed $\G_2$-structure was discovered by Awada, Duff, and Pope \cite{awadaduffpope}, and has automorphism group $\Sp(2) \Sp(1).$ 
%#v3 
Appendix \ref{appendix:squashed} includes a proof that $\phi_{sq}$ is nearly parallel.
\end{example}

\begin{rmk}
Alexandrov and Semmelmann \cite{alexandrovsemmelmannnearlyparallel} have shown that both the standard and the squashed $\G_2$-structures are rigid among nearly parallel $\G_2$-structures. These remain the only known nearly parallel structures on the 7-sphere. %, with any of its differentiable structures. %Although the exotic 7-spheres have coclosed $\G_2$-structures by the work of Crowley and Nordstrom \cite{crowleynordstromnewinvariants} (as does any manifold which admits a topological $\G_2$-structure), it is not known if they carry nearly parallel structures.
We also note that the (non-)existence of a \emph{closed} % (and necessarily not coclosed)
$\G_2$-structure on the 7-sphere is a well-known open problem.
\end{rmk}
 
\begin{rmk} 
The presence of the two distinct nearly parallel $\G_2$-structures on $S^7$ can be attributed to its 3-Sasakian structure (generated by $I_1, I_2,$ and $I_3$). Any 3-Sasakian 7-manifold carries two non-isomorphic $\G_2$-structures, one with standard fibers and one with ``squashed'' fibers; see Friedrich {\it et al.} \cite[Theorem 5.4]{friedrichkathmoroianusemmelmann} or Galicki-Salamon \cite[Proposition 2.4]{galickisalamon3sasaki}. Examples of 3-Sasakian 7-manifolds were constructed in abundance by Boyer, Galicki, Mann, and Rees \cite{boyergalickimannrees3sasaki}, giving many compact inhomogeneous nearly parallel $\G_2$-manifolds as a byproduct.
%#v3
%\begin{equation*}%\label{-4np}
%d \phi_{sq} = - 4 \psi_{sq}.
%\end{equation*}
\end{rmk}

\subsection{$\G_2$-instantons} Recall that a connection $A$ is called a \emph{$\G_2$-instanton} if its curvature satisfies (\ref{instantondef1}),
or equivalently
\begin{equation}\label{instantondef2}
\psi \wedge F_A = 0.
\end{equation}
If the $\G_2$-structure $\phi$ is nearly parallel, then from (\ref{instantondef1}) and (\ref{nearlypar}), we have
\begin{equation}\label{instareym}
\begin{split}
 0 & = D_A^*F_A - * D_A \left( \phi \wedge F_A \right) \\
 & = D_A^*F_A - * \left( \tau_0 \psi \wedge F_A -\phi \wedge D_A F_A \right) \\
 & = D_A^*F_A.
 \end{split}
 \end{equation}
We have used (\ref{instantondef2}) and the Bianchi identity in the last line. Hence, in the nearly-parallel case, any $\G_2$-instanton is Yang-Mills. %Note that it may fail to be minimizing, as is the case for any nontrivial Yang-Mills connection on the standard $S^7$ (see Bourguignon and Lawson \cite{bourguignonlawson}). % (\emph{i.e.} its second variation has negative eigenvalues). %this will be true of all the examples considered in this paper.
This observation goes back to Harland and N\"olle \cite{harlandnolle}.
%#v3 

The linearization of (\ref{instantondef2}) is
$$\psi \wedge D_A \alpha = 0,$$
for $\alpha \in \Omega^1\left( \gothg_E \right).$ Meanwhile, an infinitesimal gauge transformation  $u \in \Omega^0 \left( \gothg_E \right)$ acts by
$$u \mapsto D_A u.$$
The infinitesimal deformations of a $\G_2$-instanton $A,$ modulo gauge, therefore correspond to the first cohomology group of the following self-dual elliptic complex:
\begin{equation}\label{deformationcomplex}
\Omega^0 \left( \gothg_E \right) \stackrel{D_A}{\longrightarrow} \Omega^1 \left( \gothg_E \right) \stackrel{ \psi \wedge D_A }{\longrightarrow} \Omega^6 \left( \gothg_E \right) \stackrel{D_A}{\longrightarrow} \Omega^7 \left( \gothg_E \right).
\end{equation}
%\subsection{Deformation of instantons on nearly parallel $\G_2$-manifolds}
Folding (\ref{deformationcomplex}) and writing $d = D_A,$ we obtain the \emph{deformation operator}
\begin{equation}\label{deformationoperator}
\begin{split}
\scrL_A & :  \Omega^0\left( \gothg_E \right) \oplus \Omega^1 \left( \gothg_E \right) \to \Omega^0\left( \gothg_E \right) \oplus \Omega^1 \left( \gothg_E \right) \\
& \begin{pmatrix}
u \\
\alpha
\end{pmatrix}
\mapsto
\begin{pmatrix}
d^*\alpha \\
du + * \left( \psi \wedge d\alpha \right)
\end{pmatrix}.
\end{split}
\end{equation}
This is a first-order, self-adjoint, elliptic operator.

Squaring (\ref{deformationoperator}), we obtain
\begin{equation}\label{lausquared}
\begin{split}
\scrL_A^2  \begin{pmatrix}
u \\
\alpha
\end{pmatrix} & = 
\begin{pmatrix}
d^* d u - *\left( d \psi \wedge d \alpha + \psi \wedge F_A \wedge \alpha \right) \\
d d^* \alpha + * \left( \psi \wedge d \left( * \left( \psi \wedge d \alpha \right) + F_A \wedge u \right) \right)
\end{pmatrix} \\
& = \begin{pmatrix}
d^* d u \\
d d^*\alpha + * \left( \psi \wedge d * \left( \psi \wedge d\alpha \right) \right)
\end{pmatrix}.
\end{split}
\end{equation}
Over a compact manifold, integration by parts implies
\begin{equation*}
\ker \mathscr{L}_A = \ker \mathscr{L}_A^2.
\end{equation*}
\begin{comment}
\begin{equation*}
\mathscr{L}_A \begin{pmatrix} u \\ \alpha \end{pmatrix} = 0 \Leftrightarrow \mathscr{L}_A^2 \begin{pmatrix} u \\ \alpha \end{pmatrix} = 0.
\end{equation*}
\end{comment}
Hence, $du \equiv 0$ for any infinitesimal deformation on a compact nearly parallel $\G_2$-manifold, and $u \equiv 0$ if $A$ is irreducible.

%\begin{rmk}
We shall use the following interior product notation:
\begin{equation}\label{2formaction}
\begin{split}
 dx^{01} \rintprod \frac{\partial}{\partial x^1} = dx^0, \qquad dx^{01} \rintprod \frac{\partial}{\partial x^0}  = - dx^1, \quad  \text{{\it etc.}}
\end{split}
\end{equation}
\begin{comment}
\begin{equation}\label{2formaction}
\begin{split}
e_0 \intprod dx^{01} & = e_1,  \qquad dx^{01} \rintprod e_1 = e_0\\
e_1 \intprod dx^{01} & = - e_0, \qquad dx^{01} \rintprod e_0  = - e_1, \quad  \text{{\it etc.}}
\end{split}
\end{equation}
\end{comment}
We also take interior products between differential forms, {\it e.g.} %on Euclidean space
$$dx^{01} \rintprod dx^1 = dx^0, \qquad dx^{01} \rintprod dx^{01} = 1,$$
and similarly in general using the metric.
In particular, for a 2-form $b,$ we have
$$* \left( \psi \wedge b \right) = \phi \rintprod b.$$
%\begin{rmk}
For a $\gothg_E$-valued 1-form $\alpha,$ we shall write
\begin{equation*}
\scrL_A(\alpha) = \scrL_A \begin{pmatrix}
0 \\
\alpha
\end{pmatrix} = \begin{pmatrix}
d^*\alpha \\
\phi \rintprod d\alpha
\end{pmatrix}.
\end{equation*}
Then (\ref{lausquared}) becomes
\begin{equation}\label{lasquaredfirst}
\scrL_A^2(\alpha) = d d^*\alpha + \phi \rintprod d \left( \phi \rintprod d \alpha \right).\end{equation}
%The Weitzenbock formula for this operator is readily obtained from the following lemma.
%\end{rmk}

\begin{comment}
\begin{proof} This is an algebraic identity, and can be checked for the standard forms $\phi_0$ and $\psi_0$ above (\ref{phi0}-\ref{psi0}).  The result is immediate from (\ref{phi0}-\ref{psi0}).
\end{proof}
\end{comment}

\begin{lemma}[{\cite[(2.7-2.8)]{bryantlaplacianremarks}}]\label{lemma:phiidents} The following identities hold between any positive 3-form $\phi,$ the associated metric $g = g_\phi,$ and the dual 4-form $\psi = *_{\phi} \phi:$
\begin{equation*}
g^{pq} \phi_{p ij} \phi_{q k \ell} = g_{i k} g_{j\ell} - g_{i\ell} g_{j k} + \psi_{i j k \ell}
\end{equation*}
\begin{equation*}
g^{pq} g^{\ell m} \phi_{p\ell i} \psi_{qm jk} = 2 \phi_{ijk}.
\end{equation*}
\end{lemma}
\begin{proof} Since these are zeroth-order identities, it suffices to check them for the standard 3- and 4-form, given by (\ref{phi0}-\ref{psi0}), and the standard metric. This is easily accomplished using the fact that $\G_2$ acts transitively on orthonormal pairs of vectors.
\end{proof}

\begin{lemma}[{\it Cf.} {\cite[Lemma 7.1]{yuanqideformations}}]\label{lemma:dbstar}
For a nearly parallel $G_2$-structure and any 2-form $b,$ there holds
\begin{equation}\label{dbstar:maineq}
d\left( b \intprod \phi \right) \intprod \phi = d^*b - db \intprod \psi + \frac{\tau_0}{2} b \intprod \phi.
\end{equation}
\end{lemma}
\begin{proof} %This follows by calculating using Lemmas
Write $b = \frac12 b_{ij} dx^i \wedge dx^j$ in normal coordinates. We then have
\begin{equation*}
\begin{split}
\left( d \left( b \intprod \phi \right) \intprod \phi \right)_k & = \frac12 \left( \frac12 \left( \nabla_m (b_{ij} \phi_{ij\ell}) - \nabla_\ell (b_{ij} \phi_{ijm}) \right) \right) \phi_{m \ell k} \\
& = \frac14 \left( \nabla_m b_{ij} \phi_{ij \ell} + b_{ij} \nabla_m \phi_{ij\ell} - \nabla_\ell b_{ij} \phi_{ijm} - b_{ij} \nabla_{\ell} \phi_{ijm} \right) \phi_{m \ell k} \\
& = - \frac12 \nabla_m b_{ij} \phi_{\ell ij} \phi_{\ell mk} + \frac{\tau_0}{8} b_{ij} \psi_{m ij \ell} \phi_{m \ell k},
\end{split}
\end{equation*}
where we have used Lemma \ref{lemma:gradphi} in the last line.
By Lemma \ref{lemma:phiidents}, this becomes
\begin{equation*}
\begin{split}
\left( d \left( b \intprod \phi \right) \intprod \phi \right)_k & = - \frac12 \nabla_m b_{ij} \left( g_{im} g_{jk} - g_{ik} g_{jm} + \psi_{ijmk} \right) + \frac{\tau_0}{4} b_{ij} \phi_{ijk} \\
& = - \nabla_i b_{ik} - \frac12 \nabla_m b_{ij} \psi_{mijk} + \frac{\tau_0}{4} b_{ij} \phi_{ijk},
\end{split}
\end{equation*}
which agrees with the expression (\ref{dbstar:maineq}).
\end{proof}

\begin{prop}\label{prop:weitz} For a $\G_2$-instanton with respect to a nearly parallel $\G_2$-structure $\phi,$ we have
\begin{equation}\label{weitwithsa}
\begin{split}
\mathscr{L}_A^2(\alpha) %& = \frac{\tau_0}{2} \phi \rintprod d\alpha + \Delta_A \alpha - \LB F_A \rintprod \alpha \RB \\
& = \frac{\tau_0}{2} \phi \rintprod d \alpha + \mathscr{S}_A \left( \alpha \right),
\end{split}
\end{equation}
where
$$\mathscr{S}_A\left( \alpha \right) = \nabla_A^*\nabla_A \alpha + \Ric(\alpha) - 2 \LB F_A \rintprod \alpha \RB$$
is the Yang-Mills stability operator (see Bourguignon-Lawson \cite{bourguignonlawson}).
In the case of the round $7$-sphere, we have
\begin{equation}\label{LAsquared}
\begin{split}
\mathscr{L}_A^2(\alpha ) & = 2 \phi \rintprod d \alpha + \nabla_A^*\nabla_A \alpha + 6 \alpha - 2 \LB F_A \rintprod \alpha \RB .
\end{split}
\end{equation}
\end{prop}
\begin{proof} From (\ref{lasquaredfirst}) and Lemma \ref{lemma:dbstar}, we have
\begin{equation}\label{weitzmess}
\begin{split}
\scrL_A^2(\alpha) & = d d^*\alpha + \phi \rintprod d \left( \phi \rintprod d\alpha \right) \\
& = d d^*\alpha + d^* d \alpha - d^2 \alpha \rintprod \psi + \frac{\tau_0}{2} \phi \rintprod d \alpha \\
& = \left( d d^* + d^* d \right) \alpha - \left( F_A \wedge \alpha \right) \intprod \psi + \frac{\tau_0}{2} \phi \rintprod d \alpha \\
& = \frac{\tau_0}{2} \phi \rintprod d \alpha + \left( d d^* + d^* d \right) \alpha + \left( F_A \intprod \psi \right) \rintprod \alpha.
\end{split}
\end{equation}
But since $A$ is an instanton, we have
$$F_A \intprod \psi = * \left( F_A \wedge \phi \right) = - F_A.$$
Substituting into (\ref{weitzmess}) and applying the Bochner formula yields (\ref{weitwithsa}). Then (\ref{LAsquared}) is obtained by substituting $\tau_0 = 4$ and $\Ric_g = 6g$ on the round 7-sphere.
\end{proof}

\begin{rmk}\label{rmk:instantons}
%Instantons on nearly parallel $\G_2$-manifolds have been studied by Harland and N{\"o}lle \cite{harlandnolle}
Ball and Oliveira \cite{balloliveira} have studied instantons on the Aloff-Wallach spaces, which are nearly parallel. Singhal \cite{raginisinghal} also studies instantons on homogeneous nearly parallel $\G_2$-manifolds, using spinorial methods similar to those of Charbonneau and Harland \cite{charbharland} in the context of nearly K{\"a}hler manifolds. %, obtaining results similar to those of this paper for the standard instanton (described in the next section).
%#v3
\end{rmk}

\section{Hopf fibration and standard instantons}\label{sec:quatinst}

%We review the essential 4D constructions for this work, using 

In this section, we give an explicit description of the standard (A)SD instanton and its $\G_2$ relative. We shall use a variant of Atiyah's quaternionic notation \cite{atiyahgeometry} based on the convention (\ref{quaternionbasis}):
\begin{equation*}
\begin{split}
x & = x^0 \mathbf{1} - x^1 \bi - x^2 \bj - x^3 \bk, \qquad \qquad \,\,\,\, \bar{x} = x^0 \mathbf{1} + x^1 \bi + x^2\bj + x^3 \bk  \\
 dx & = dx^0 \mathbf{1} - dx^1 \bi - dx^2 \bj - dx^3 \bk , \qquad d\bar{x}  = dx^0 \mathbf{1} + dx^1 \bi + dx^2 \bj + dx^3 \bk.
\end{split}
\end{equation*}
Here $x$ and $\bar{x}$ are $\H$-valued functions, and $dx$ and $d\bar{x}$ are $\H$-valued differential forms on $\R^4.$ We define $y$ and $\bar{y}$ similarly, and will identify
\begin{equation*}
\R^8 = \H_x \oplus \H_y
\end{equation*}
as above.

\subsection{Quaternionic Hopf fibration}\label{ss:hopf} % and standard 4D instantons}
The (right) Hopf fibration is given by the quotient projection under right-multiplication by $\H^{\times}:$
\begin{equation}\label{hopf}
\pi: \H^2 \setminus \{ (0,0) \} \rightarrow \H \P^1.
\end{equation}
The fibration (\ref{quathopf}) is obtained by restricting (\ref{hopf}) to $S^7,$ giving the quotient projection under right-multiplication by $\Sp(1) \cong S^3 \subset \H^\times.$

To see the identification $\H \P^1 \cong S^4$
explicitly, observe that $\Sp(2)$ acts on $S^7$ by isometries commuting with $\Sp(1)_r.$ The stabilizer of an $S^3$ fiber of (\ref{quathopf}) is the subgroup $\Sp(1)_l \times \Sp(1)_l \subset \Sp(2).$ Meanwhile, $\Sp(2)$ acts by conjugation on the 5-dimensional space of  $2\times 2$ traceless self-adjoint quaternionic matrices:
\begin{equation}\label{wdef}
W = \left\{ \left( \begin{array}{cc}
a  & \bar{z} \\
z & -a  \end{array} \right) \mid a \in \R, z \in \H \right\},
\end{equation}
where the stabilizer of an axis is again $\Sp(1)_l \times \Sp(1)_l.$ We therefore have a map
\begin{equation*}\label{quotientid}
\H \P^1 = \Sp(2)/ \Sp(1)_l \times \Sp(1)_l \,\, \tilde{\longrightarrow} \,\, \SO(5) / \SO(4) = S^4
 \end{equation*}
which is an isometry, up to a factor of $1/2.$ %The factor of $1/2$ follows by observing that a great circle in $S^7$ is mapped 2-to-1 onto a great circle in $S^4.$ %the factor must be $1/2.$

\begin{comment}

We claim that the tautological bundle over $\H \P^1$ is the positive spin bundle $S^+ \to S^4,$ and its orthogonal complement is $S^-.$ % This in fact follows because these are the bundles associated to each factor of $\Sp(1) \times \Sp(1) = \Spin(4)$ under the lift of the principal $\SO(4)$-bundle given by the identification (\ref{quotientid}).
To check this, it suffices to write down the Clifford multiplication $\gamma : TS^4 \to \mathrm{Hom} \left( S^+ , S^- \right),$ satisfying
\begin{equation}\label{cliffordrule}
\gamma(v)^* \gamma(v) = |v|^2.
\end{equation}
%for all tangent vectors $v.$ 
But a tangent vector $v \in T_x\H \P^1$ is exactly an $\H_r$-linear map from the tautological bundle to its complement, {\it i.e.}, $\gamma(v): S^+_x \to S^-_x.$ Such a map is necessarily given by multiplication by a quaternion on the left, which automatically satisfies (\ref{cliffordrule}).
With the orientation convention on $\H \P^1$ coming from (\ref{quaternionbasis}), one checks that $\gamma$ sets up an isomorphism
\begin{equation*}%\label{selfdualS+}
\Lambda^2_+(T^*S^4) \tilde{\longrightarrow} \End_0 \, S^+
\end{equation*}
%checked on 8/17/19
as required.

\end{comment}

The basic link between instantons on $S^7$ and $S^4$ is as follows; a more general statement appears in Proposition \ref{prop:hym} below.

\begin{lemma}\label{lemma:instapullback} Let $B$ be a connection on a principal bundle over $S^4.$ Then $B$ is an ASD instanton if and only if the pullback $A = \pi^*B$ by the Hopf fibration is a $\G_2$-instanton, for either the standard or the squashed nearly parallel $\G_2$-structure on $S^7.$ %This bundle is nontrivial for $c_2 = ?$.
\end{lemma}
\begin{proof} %This follows from the discussion in \S and the twistor construction, but we give a direct proof.
By $\Sp(2)$-invariance both of $\phi_{std}$ and of the fibration, it suffices to consider the point $(1,0).$ %$S^3_0$
%which is contained in $\R^4_x$ (??).
From (\ref{phistdfirstdef}), we have
\begin{equation}\label{phis30}
 \phi (x,0)  = \nu^x + \zeta_1^x \wedge \omega_1^y + \zeta_2^x \wedge \omega_2^y - \zeta_3^x \wedge \omega_3^y.
\end{equation} %\right|_{S^3_0} \left.
and
\begin{equation}\label{phiat10}
\phi(1,0) = dx^{123} + dx^1 \wedge \omega_1^y + dx^2 \wedge \omega_2^y - dx^3 \wedge \omega_3^y.
\end{equation}
The orthogonal complement of the fiber through $(1,0)$ is $\R^4_y,$ which is mapped conformally onto the tangent space of $S^4$ at $p = \pi(1,0).$ Hence, $F_A(1,0) = \pi^*F_B(p)$ is equal to a 2-form on $\R^4_y.$ It is clear from the expression (\ref{phiat10}) that $\phi \rintprod F_A (1,0)$ vanishes if and only if this 2-form is ASD, which is equivalent to the same statement for $F_B(p).$

The same argument applies on $\phi_{sq}.$
\end{proof}

\subsection{Standard (A)SD instanton} Let $P^+$ denote the principal $\Sp(1)$-bundle associated to the right Hopf fibration. We define the \emph{standard self-dual instanton} to be the connection on $P^+$ induced by the round metric on the total space. The corresponding connection form on $P^+$ is the $\mathfrak{su}(2)$-valued 1-form
\begin{equation*}%\label{stdasdconnform}
\frac{ \Im \LB \bar{x} \, d x + \bar{y} \, d y \RB }{ |x|^2 + |y|^2 }.
\end{equation*}
Pulling back to $\R^4$ by the map $x \mapsto (x,1),$ %$x \to \LB x,1 \RB \in \H \P^1_r,$
we obtain the well-known connection matrix % on $\R^4:$
\begin{comment}  instanton on $\R^4:$
\begin{equation*}\label{standardsdinst}
\frac{\Im  \bar{x} dx }{1 + |x|^2}.
\end{equation*}
Its curvature is the $\mathfrak{su}(2)$-valued self-dual 2-form
\begin{equation*}
\frac{ d\bar{x} \wedge d x}{\left( 1 + |x|^2 \right)^2 }. % = \frac{ i \otimes \eta_1 + j \otimes \eta_2 + k \otimes \eta_3 }{\left( 1 + |x|^2 \right)^2 }.
\end{equation*}
\end{comment}
%standard anti-self-dual (ASD)
\begin{equation*}%\label{stdinstaconn}
\frac{\Im  \bar{x} d x }{1 + |x|^2},
\end{equation*}
whose curvature is the $\mathfrak{su}(2)$-valued self-dual 2-form
\begin{equation*}%\label{stdinstacurvature}
\frac{ d \bar{x} \wedge d x }{\left( 1 + |x|^2 \right)^2 }. % = \frac{ i \otimes \eta_1 + j \otimes \eta_2 + k \otimes \eta_3 }{\left( 1 + |x|^2 \right)^2 }.
\end{equation*}
See Atiyah \cite[\S 1]{atiyahgeometry} for these formulae, as well as a generalization giving the complete ADHM construction.

Similarly, we define the \emph{standard anti-self-dual (ASD) instanton} $P^-$ to be the principal bundle associated to the left Hopf fibration, with connection form
\begin{equation}\label{stdasdconnform}
\frac{ \Im \LB w\, d \bar{w} + z \, d\bar{z} \RB }{ |w|^2 + |z|^2 }.
\end{equation}
In the stereographic chart on $\R^4,$ this has a connection matrix
\begin{equation}\label{stdinstaconnmatrix}
B_0(x) = \frac{\Im x d \bar{x} }{1 + |x|^2}
\end{equation}
and curvature the $\Im \H$-valued self-dual 2-form
\begin{equation}\label{stdinstacurvature}
F_{B_0}(x) = \frac{ d x \wedge d \bar{x} }{\left( 1 + |x|^2 \right)^2 }. % = \frac{ i \otimes \eta_1 + j \otimes \eta_2 + k \otimes \eta_3 }{\left( 1 + |x|^2 \right)^2 }.
\end{equation}
 % on $\R^4.$ %that include all finite-energy $\SU(2)$-instantons on $\R^4.$

\subsection{Standard $\G_2$-instanton}\label{ss:stdg2inst}
Let
\begin{equation}
P = P^-\times_{S^4} P^+ \stackrel{\pi_2}{\longrightarrow} S^7
\end{equation}
be the fiber product over $S^4$ of $P^-$ with $P^+,$ considered as a principal $\Sp(1)$-bundle via the projection to $P^+ = S^7.$
According to \S \ref{ss:hopf} and Lemma \ref{lemma:instapullback}, the pullback of the standard ASD instanton on $P^-$ is a $\G_2$-instanton on $P \to S^7,$ which we call the \emph{standard $\G_2$-instanton}, $A_0.$

We may obtain a connection matrix for $A_0$ by pulling back the connection form of the standard ASD instanton (\ref{stdasdconnform}) %on $\H^2_l \setminus 0$
by the fiber-preserving map $\H^{\times 2}_r \to \H^{\times 2}_l $ given by
\begin{equation*}
\begin{split}
w & = y^{-1} \\
z & = x^{-1}.
\end{split}
\end{equation*}
\begin{comment}
\begin{equation*}
\begin{split}
\begin{pmatrix}
x \\
y
\end{pmatrix} & \mapsto  \begin{pmatrix}
y^{-1} \\
x^{-1}
\end{pmatrix}.
\end{split}
\end{equation*}
\end{comment}
This gives
\begin{equation*}
\begin{split}
A_0(x,y) & = \frac{1}{ |x|^{-2} + |y|^{-2} } \Im \LB x^{-1} d \left( \bar{x}^{-1} \right) + y^{-1} d \left( \bar{y}^{-1} \right) \RB \\
& = \frac{|x|^2|y|^2}{ |x|^{2} + |y|^{2} } \Im \LB - x^{-1} \bar{x}^{-1} d \bar{x} \bar{x}^{-1} - y^{-1} \bar{y}^{-1} d\bar{y} \bar{y}^{-1}  \RB \\
%& = \frac{-1}{|x|^2 + |y|^2} \Im \LB |y|^2 d\bar{x} \, \bar{x}^{-1} + |x|^2 d\bar{y} \, \bar{y}^{-1} \RB \\
& = \frac{1}{|x|^2 + |y|^2} \Im \LB |y|^2  x^{-1} \, dx + |x|^2 y^{-1} \, dy \RB,
\end{split}
\end{equation*}
which is (\ref{standardg2instanton}) above.
The singularity along the $x$-axis can be removed by applying the gauge transformation $g(x,y) = y/|y|:$ noting that
$dg g^{-1} = - \Im y d\bar{y} / |y|^2,$ we obtain
\begin{equation*}
\begin{split}
 g(A_0) = gA_0 \bar{g} - dg \bar{g} & = \frac{1}{|x|^2 + |y|^2} \Im \LB  y  x^{-1} dx \bar{y} + \frac{|x|^2}{|y|^2} dy \bar{y} + \frac{|x|^2 + |y|^2}{|y|^2} y d\bar{y} \RB \\
% & = \frac{1}{|x|^2 + |y|^2} \Im \LB  y  x^{-1} dx \bar{y} + \frac{|x|^2}{|y|^2} dy \bar{y} - \frac{|x|^2 }{|y|^2} dy \bar{y} + y d\bar{y} \RB \\
 & = \frac{1}{|x|^2 + |y|^2} \Im \LB y x^{-1} dx \bar{y} + y d\bar{y} \RB.
\end{split}
\end{equation*} %#v3
 The singularity along the $y$-axis can be removed similarly. %by applying the gauge transformation $g(x,y) = \bar{y}/|y|.$
However, %according to the following proposition,
they cannot be removed simultaneously, for according to the following proposition, the bundle $P \to S^7$ is nontrivial.
Recall that by the clutching construction, topological $\SU(2)$-bundles on $S^7$ are classified by $\pi_6(S^3) = \Z_{12}.$

%#v3

\begin{prop}[Crowley-Goette {\cite[(1.18)]{crowleygoettekreck}}]
For an $SU(2)$-bundle on $S^4$ with $c_2(E) = \kappa,$ the pullback bundle %$\pi^*E$
on $S^7$ has homotopy class
\begin{equation*}
\frac{\kappa \left( \kappa + 1 \right)}{2} \in \Z_{12}. %= \pi_6(S^3).
\end{equation*}
\end{prop}

\subsection{Curvature calculation}\label{ss:curvcalc} We check directly that $A_0$ is a $\G_2$-instanton on $S^7.$ By (\ref{phistdzerothdef}), this is equivalent to showing that (\ref{standardg2instanton}) is a $\Spin(7)$-instanton on $\R^8.$ We calculate %May 2018
\begin{equation*}
\begin{split}
d \left( x^{-1} dx \right) %& = - dx \wedge (-x^{-1} dx x^{-1} ) \\
& = - x^{-1} dx  \wedge  x^{-1} dx
\end{split}
\end{equation*}
and
\begin{equation*}
\begin{split}
d \left( \frac{ |y|^2}{|x|^2 + |y|^2} \right) 
& = \frac{ |x|^2|y|^2}{\left( |x|^2 + |y|^2 \right)^2} 2 \Re \LB y^{-1} dy - x^{-1} dx \RB \\
& = - \, d \left( \frac{|x|^2}{|x|^2 + |y|^2} \right).
\end{split}
\end{equation*}
This gives
\begin{equation}\label{dacalc}
\begin{split}
dA_0 & = \frac{1}{\left( |x|^2 + |y|^2 \right)^2} \Im \LB \begin{split} 
&  x^{-1} dx \wedge \left(-2 |x|^2|y|^2 \Re \LB y^{-1} dy - x^{-1} dx \RB - \left( |y|^4 + |x|^2 |y|^2 \right) x^{-1} dx  \right)   \\
& + y^{-1} dy \wedge \left(-2 |x|^2 |y|^2 \Re \LB x^{-1} dx  - y^{-1} dy \RB - \left( |x|^4 + |x|^2 |y|^2 \right) y^{-1} dy \right) \end{split} \RB \\
& = \frac{1}{\left( |x|^2 + |y|^2 \right)^2} \Im \LB \begin{split} 
& x^{-1} dx \wedge \left(-2 |x|^2|y|^2 \Re y^{-1} dy  + |y|^2 d\bar{x} x  -  |y|^4 x^{-1} dx \right)  \\
& + y^{-1} dy \wedge \left( - 2 |x|^2 |y|^2 \Re  x^{-1} dx + |x|^2 d\bar{y} y  -  |x|^4 y^{-1} dy \right) \end{split} \RB \\
& = \frac{1}{\left( |x|^2 + |y|^2 \right)^2} \Im \LB \begin{split} 
&  |y|^2 x^{-1} dx  \wedge d\bar{x} x +  |x|^2 y^{-1} dy \wedge d\bar{y} y - \bar{x} dx \wedge d \bar{y}  y  - \bar{y} dy  \wedge d\bar{x} x \\
& - 2 |x|^2 |y|^2  y^{-1} dy \wedge x^{-1} dx -  |x|^4 y^{-1} dy \wedge y^{-1} dy - |y|^4 x^{-1}  dx \wedge x^{-1} dx \end{split} \RB.
\end{split}
\end{equation}
On the other hand, we have
\begin{equation}\label{asquaredcalc}
A_0 \wedge A_0 = \frac{1}{\left( |x|^2 + |y|^2 \right)^2} \Im \LB |y|^4 x^{-1} dx \wedge x^{-1} dx + |x|^4 y^{-1} dy \wedge y^{-1} dy + 2 |x|^2 |y|^2 x^{-1} dx \wedge y^{-1} dy  \RB.
\end{equation}
Adding (\ref{dacalc}) and (\ref{asquaredcalc}), we obtain the curvature form
\begin{equation}\label{stdg2instcurv}
\begin{split}
F_{A_0} (x,y) & = dA_0 + A_0 \wedge A_0 \\
& = \frac{1}{\left( |x|^2 + |y|^2 \right)^2} \Im \LB
 |y|^2 x^{-1} dx \wedge d\bar{x} x +  |x|^2 y^{-1} dy \wedge d\bar{y} y - 2 \bar{x} dx \wedge d \bar{y} y  \RB \\
& = \frac{1}{ \left(|x|^2 + |y|^2 \right)^2 |x|^2 |y|^2 } \Im \LB \overline{\left( d\bar{x} x |y|^2 - d \bar{y} y |x|^2 \right)} \wedge \left(d\bar{x} x |y|^2 - d \bar{y} y |x|^2 \right) \RB.
\end{split}
\end{equation}
The 2-forms $dx \wedge d\bar{x},$ $dy \wedge d\bar{y},$ and $dx \wedge d \bar{y}$ are each invariant under $\Sp(1)_r.$ Therefore, the 2-form part of $F_{A_0}$ lies in $\Lie(\Sp(2)) \subset \Lie(\Spin(7)),$ as claimed.

\begin{comment}
\begin{rmk}
If we apply the gauge transformation $g(x,y) = \bar{y}/|y|,$ as before, we get
\begin{equation}\label{stdg2curvatureingauge}
\begin{split}
 F_{g(A_0)} = g F_{A_0} \bar{g} (x,y) & = \frac{1}{\left( |x|^2 + |y|^2 \right)^2} \Im \LB
\bar{y} x d\bar{x} \wedge dx x^{-1} y  +  |x|^2 d\bar{y} \wedge dy - 2 d \bar{y} \wedge dx \bar{x} y \RB.
\end{split}
\end{equation}
Notice that in this gauge, the curvature along the $x$-axis is given by
\begin{equation}
\begin{split}
 F_{g(A_0)} (x,0) & = \frac{
d\bar{y} \wedge dy }{ |x|^2 }.
\end{split}
\end{equation}
Lemma \ref{lemma:eigenvalues} and the Weitzenbock formula (\ref{LAsquared}) can be used to show that all deformations of the standard instanton are of the form $ F_{A_0} \rintprod X$ for vector fields $X$ on $S^7.$ 
\end{rmk} 
\end{comment}

\subsection{Linear deformations}\label{ss:lineardefos}

Let $A$ be a conical instanton on $\R^8 \setminus \{0\}$ %{\it i.e.}, a pullback from $S^7.$ 
whose curvature $F_A$ takes values in $\Lie(\Sp(2)) \otimes \gothg_E.$

Given any $8\times 8$ matrix $M,$ we associate the vector field
\begin{equation*}
X_M = M \cdot \vec{r} = M^i{}_{j} x^j \frac{\partial}{\partial x^i}
\end{equation*}
on $\R^8,$ as well as the $\gothg_E$-valued 1-form
\begin{equation}\label{lineardeformation}
\alpha_M = X_M \intprod F_{A} \in \Omega^1 \left( \gothg_E \right).
\end{equation}
%We shall call such $\alpha_M$ a \emph{linear deformation of $A$} if it lies in the kernel of $\scrL_{A}.$
Notice that (\ref{lineardeformation}) corresponds to pushforward by the diffeomorphism generated by $X_M,$ with its horizontal lift to the bundle. For, working in a local gauge where $A_{X_M} = 0,$ we have
$$\frac{d}{dt} \exp(-tX_M)^* A = X_M \left( A \right) = F(X_M, -) = \alpha_M.$$ The action on the curvature is given by
\begin{equation}\label{deformationexpthing}
\frac{d}{dt} \exp(-tX_M)^* F_A = \frac{d}{dt} F_{\exp^*(-t X_M) A}  = D_{A} \alpha_M.
\end{equation}

\begin{lemma}\label{lemma:coulomb} For $M \in \Lie(\Sp(2))^\perp \subset\R^{8\times 8},$ we have
\begin{equation}
\left( D_{A}^{S^7} \right)^* \alpha_M = 0.
\end{equation}
% In particular, we $\vec{r} \intprod F_A = 0.$
\end{lemma}
\begin{proof} Let $\alpha = \alpha_M$ and $F = F_A.$ In coordinates on $\R^8,$ we have
\begin{equation}\label{R8coulomb}
\left( D_A^{\R^8} \right)^*\alpha = \nabla^i X^j F_{ij} + X^j \nabla^i F_{ij} = M_{ij} F_{ij}.
\end{equation}
Since the curvature takes values in $\Lie(\Sp(2))$ and $M$ belongs to the orthogonal complement, the expression (\ref{R8coulomb}) vanishes identically.
The result then follows from the formula
$$\left( D_A^{\R^8} \right)^*\alpha = \left( D_A^{S^7} \right)^*\alpha - \LA \hat{r}, \nabla_{\hat{r}} \left( \alpha \left( \hat{r} \right) \right) \RA$$
and the fact that $\alpha(\hat{r}) \equiv 0$ for a conical instanton.
\end{proof}

\begin{prop}\label{prop:15deformations} Let $W$ be the 5-dimensional space (\ref{wdef}). % of all $\H_r$-linear, traceless, self-adjoint maps of $\R^8.$
Then $\ker \scrL_A$ contains the space %of linear deformations
\begin{equation}\label{15deformations}
\mathscr{V}_A = \{ \alpha_M \mid M \in W \oplus W I_1 \oplus W I_2 \},
\end{equation}
where $\alpha_M$ is defined by (\ref{lineardeformation}).

If $F_A(x,y)$ spans $\Lie(\Sp(2))$ as $(x,y)$ varies over $\R^8,$ then $\dim(\scrV_A) = 15.$ %#v3
\end{prop}
\begin{proof} %(\emph{Sketch})
According to (\ref{deformationexpthing}), %the space (\ref{15deformations}) must be contained in $\ker \scrL_A.$ For, 
the subspace $\{\alpha_M \mid M \in W\}$ corresponds to pushforward by elements of $\SL(2, \H),$ which preserve the algebra $\Lie(\Sp(2));$ hence, these correspond to infinitesimal deformations. By Lemma \ref{lemma:coulomb}, the space (\ref{15deformations}) is in Coulomb gauge, so the first factor lies in the kernel of $\scrL_A.$ % is precisely the 5-dimensional space of deformations coming from the conformal transformations of $S^4,$ which preserve the algebra $\Lie(\Sp(2)),$ and therefore lie in the kernel of $\scrL_A.$ %(see Appendix \ref{appendix:ahs}). %, and consists of symmetric matrices and closed vector fields on $S^7.$

The second factor may be described as follows:
\begin{equation*}
WI_1 = \left\{ \left( \begin{array}{cc}
aI_1  & \bar{z}I_1 \\
zI_1 & -aI_1  \end{array} \right) \mid a \in \R, z \in \H \right\}.
\end{equation*}
The matrix $\left( \begin{array}{cc}
I_1  & 0 \\
0 & -I_1  \end{array} \right)$ is just the first element in (\ref{omegad123}), which belongs to the subspace $\Lambda^2_d \subset \Lie(\Spin(7)).$ Since $I_1$ commutes with $\H_l,$ we also have
\begin{equation*}
\left( \begin{array}{cc}
0 & \bar{z}I_1 \\
zI_1 & 0 \end{array} \right) = \left( \begin{array}{cc}
0 & - \overline{zI_1}\\
zI_1 & 0 \end{array} \right),
\end{equation*}
which belongs to the subspace $\mathfrak{F}_{x,y} \subset \Spin(7)$ given by (\ref{frakFdef}). Hence, the second factor corresponds to pushforward by elements of $\Spin(7),$ which are in Coulomb gauge by Lemma \ref{lemma:coulomb}. The same is true of the third factor. %represents an infinitesimal deformation. %This is a subspace of $\Lie(\Spin(7)),$ so these also correspond to infinitesimal deformations, which 
%Meanwhile, according to the decomposition (\ref{spin7splitting}), we have
%$$W I_1 \oplus W I_2 \subset \Lambda^2_d \oplus \mathfrak{F}_{x,y} \subset \Lie \left( \Spin(7) \right).$$ %containing skew-symmetric matrices and coclosed vector fields on $S^7,$
%These correspond to $\Spin(7)$-rotations, and are also contained in the kernel of $\scrL_A.$ %which is the normalizer of $\Lie(\Sp(2)).$ % \subset \Lie(\Spin(7)).$ %The latter acts by gauge on $A_0,$ by construction.
 %these % act nontrivially on $F_A$ 
%are nonzero and linearly independent from $W.$
Therefore the space $\scrV_A$ is contained in the kernel of the deformation operator $\scrL_A.$

For the second statement, if the curvature $F_A(x,y)$ spans $\Lie(\Sp(2)),$ then for any nonzero $M,$ the element $\alpha_M$ must be nonzero. In particular, the map $M \mapsto \alpha_M$ has vanishing kernel, hence rank 15.
\end{proof}

\begin{rmk} The master's thesis of Jurke \cite{jurkemasters} studies $\Spin(7)$-instantons using a quaternionic approach similar to that of this section. The author thanks an anonymous referee for pointing out this reference.
\end{rmk} %#v3

\section{Infinitesimal deformations} %on the standard 7-sphere}
\label{sec:generaldefos}

In this section, we calculate the space of infinitesimal deformations, as a $\G_2$-instanton, of the pullback to $S^7$ of a general irreducible ASD instanton on $S^4.$ We write
$$\nabla = \nabla^{S^7} = \pi_{S^7} \circ \nabla^{\R^8}$$
for the Levi-Civita connection on $S^7_{std},$ which we shall couple to the connection on any auxiliary bundle. We shall also write
\begin{equation}\label{nablavh}
\nabla^v = \nabla_{\pi_v}, \qquad \nabla^h = \nabla_{\pi_h}.
\end{equation}
Here, $\pi_v$ is the orthogonal projection to the the vertical tangent space of the Hopf fibration and $\pi_h$ is the complementary projection, with respect to the round metric. %#v3.

\subsection{Vertical and horizontal components}
Let $\Omega^1_h$ be the annihilator of vertical vector fields along the Hopf fibration, and $\Omega^1_v$ its orthogonal complement. We have
\begin{equation*}
\Omega^1_{S^7} = \Omega^1_v \oplus \Omega^1_h.
\end{equation*}
Letting
$$\Omega^{(p,q)} = \Lambda^p \Omega^1_v \otimes \Lambda^q \Omega^1_h \subset \Omega^{p + q}_{S^7},$$
we have a decomposition
$$\Omega^{k}_{S^7} = \bigoplus_{p + q = k} \Omega^{(p,q)}.$$
An element of $\Omega^{(p,q)}$ will be referred to as a $(p,q)$-form.

Let $\nu$ denote the $(3,0)$ volume form of the Hopf fibration, and let
$$\bar{\nu} = * \nu.$$
The $(0,2)$-forms split as
$$\Omega^{(0,2)} = \Omega^{2+}_h \oplus \Omega^{2-}_h$$
where $\Omega^{2\pm}_h$ are the (anti-)self-dual components %of the $(0,2)$-forms
with respect to the $(0,4)$ volume form $\bar{\nu}.$
 
For a $(0,1)$-form $b,$ we shall write $d^v b$ for the $(1,1)$ part of $db$ and $d^h b$ for the $(0,2)$ part. A similar notation will be used for $(1,0)$-forms (see Lemma \ref{lemma:dalpha} below).

\begin{defn}\label{defn:coframes} Let $\omega_i^x, \omega_i^y, \zeta_i^x,$ and $\zeta_i^y$ be as in (\ref{2formbases}) and (\ref{zetaidef}). For $i = 1,2,3,$ define the global $\Sp(2)$-invariant forms on $S^7$:
\begin{equation*}
\begin{split}
\zeta_i & = \vec{r} \intprod \left( \omega^x_i + \omega^y_i \right) \\
\omega_i^\circ & = \left. \omega^x_i + \omega^y_i \right|_{S^7} \\
\bar{\omega}_i  & = \omega_i^\circ - \frac{1}{2}\epsilon_{ijk} \zeta_j \wedge \zeta_k.
%\nu & = \zeta_1 \wedge \zeta_2 \wedge \zeta_3.
\end{split}
\end{equation*}
Notice that $\{\zeta_i\}$ and $\{ \bar{\omega}_i \}$ are global frames for $\Omega^1_v$ and $\Omega^{2+}_h,$ respectively. The vertical volume form is given by
$$\nu = \zeta_1 \wedge \zeta_2 \wedge \zeta_3.$$
From (\ref{phistdfirstdef}) and the $\Sp(2)$-invariance, we may re-express $\phi_{std}$ as follows: %, using the frames given by Definition \ref{defn:coframes}:
\begin{equation}\label{phistdsp2def}
\begin{split}
\phi_{std} 
& = \stackrel{ (3,0) }{\overbrace{ \nu} } + \stackrel{ (1,2) }{\overbrace{ \zeta_1 \wedge \bar{\omega}_1 + \zeta_2 \wedge \bar{\omega}_2 - \zeta_3 \wedge \bar{\omega}_3 }}.
\end{split}
\end{equation}
\end{defn}
\noindent We now derive the basic properties of these frames, and use them to decompose the Laplace operator into horizontal and vertical parts.

\begin{lemma}\label{lemma:gradgammai} The frame $\{\zeta_i\}$ is coclosed, and satisfies
\begin{equation}\label{gradgammai:gradvh}
\nabla^v \zeta_i = \frac{1}{2} \epsilon_{ijk} \zeta_j \wedge \zeta_k, \qquad \nabla^h \zeta_i = \bar{\omega}_i
\end{equation}
\begin{equation}\label{gradgammai:laplace}
\nabla^* \nabla^v \zeta_i = 2 \zeta_i, \qquad \nabla^* \nabla^h \zeta_i = 4 \zeta_i.
\end{equation}
Here $\nabla^v, \nabla^h$ are defined by (\ref{nablavh}) above.
\end{lemma}
\begin{proof} From the Definition \ref{defn:coframes} and (\ref{zetaidef}), for $X, Y \in T S^7,$ we have
\begin{equation}\label{gradgammai:grad}
\left( \nabla_X \zeta_i \right)  (Y) = \left( \nabla^{\R^8}_X \zeta_i \right)(Y) = \left( \omega^x_i + \omega^y_i \right)(X,Y) = \omega^\circ_i (X, Y).
\end{equation}
Coclosedness of $\zeta_i$ and (\ref{gradgammai:gradvh}) follow directly from (\ref{gradgammai:grad}).

Since $\zeta_i$ is coclosed and equal to the restriction of a linear form, we have
\begin{equation}\label{gradgammai:spherelaplace}
\nabla^*\nabla_{S^3} \zeta_i = 2 \zeta_i, \qquad \nabla^*\nabla_{S^7} \zeta_i = 6 \zeta_i
\end{equation}
as can be verified directly from (\ref{gradgammai:grad}). Then (\ref{gradgammai:laplace}) follows from (\ref{gradgammai:spherelaplace}) and the fact that $\nabla = \nabla^v + \nabla^h.$
\end{proof}

\begin{lemma}\label{lemma:dalpha}
Let $\alpha = a + b$ be a 1-form on $S^7,$ with $a = f_i \zeta_i \in \Omega^1_v $ and $b \in  \Omega^1_h .$ Then
\begin{equation*}
\begin{split}
d \alpha = \,\, \stackrel{ (2,0) }{\overbrace{ d^v a }} & + \stackrel{ (1,1)}{ \overbrace{d^h \! f_i \wedge \zeta_i } } + \stackrel{ (0,2) }{\overbrace{2 f_i \bar{\omega}_i }} \\
& + d^v b \qquad + \, d^h b. \end{split}
\end{equation*}
\end{lemma}
\begin{proof} This is a standard decomposition result, which can be seen directly as follows. By Lemma \ref{lemma:gradgammai}, we have
\begin{equation*}
d \zeta_i = 2 \omega_i^\circ, \quad d^v \zeta_i = \epsilon_{ijk} \zeta_j \wedge \zeta_k, \quad i = 1,2,3.
\end{equation*}
Therefore
\begin{equation}\label{dgammai}
d \zeta_i = \stackrel{ (2,0) }{\overbrace{ \epsilon_{ijk} \zeta_j \wedge \zeta_k }} + \stackrel{ (0,2) }{\overbrace{ 2 \bar{\omega}_i. }}
\end{equation}
%In particular, $d^h \zeta_i = 0$ for $i = 1,2,3.$
The result follows directly from (\ref{dgammai}).
\end{proof}

\begin{lemma}\label{lemma:coclosed} If $a = f_i \zeta_i$ is (co)closed on each $S^3$ fiber, then $\left( \nabla^*\nabla^h  f_i \right) \zeta_i$ is again fiberwise (co)closed.
\end{lemma}
\begin{proof} Let $U_j$ be the dual vector field of $\zeta_j,$ for $j = 1,2,3.$ The Killing vector fields $U_j$ commute with the operators $\nabla^*\nabla$ and $\nabla^*\nabla^v,$ hence also with $\nabla^* \nabla^h = \nabla^*\nabla - \nabla^*\nabla^v.$ Since $d^v a$ and $(d^v)^* a = d^*a$ are determined by $U_j (f_i),$ we conclude that $\nabla^* \nabla^h$ preserves (co)closedness on the fibers.
\end{proof}

\begin{comment}
Let $\{e_j\}_{j = 1}^7$ be a basis of vector fields which is orthonormal and satisfies $\nabla_{e_j} e_k (p)= 0$ at a given point $p,$ with
\begin{equation}
e_j(p) \in T_v S^7, \quad j = 1,2,3.
\end{equation}
\end{comment}

%We shall need the following decomposition of the Laplace operator with respect to the Hopf fibration.

\begin{prop}\label{prop:laplacev} Let $\alpha = a + b$ be a 1-form as above, where $a = f_i \zeta_i$ and $b \in \Omega^1_h.$ The vertical component of the Laplacian on $S^7$ is given by
\begin{equation}\label{laplacev:maineq}
\left( \nabla^* \nabla \alpha \right)^v =  \nabla^* \nabla^v a + 4a + \left(\nabla^* \nabla^h \! f_i  + 2\LA d^h b, \bar{\omega}_i  \RA \right)\zeta_i.
\end{equation}
Here $\nabla^* \nabla^v$ denotes the Laplacian on the $S^3$ fiber.
\end{prop}
\begin{comment}
\begin{rmk} 
Here the inner product as 2-tensors is intended:
$$\LA \omega_i , \nabla b \RA = \LA \omega_i(e_j, e_k), \nabla_{e_j} b (e_k ) \RA.$$
\end{rmk}
\end{comment}
\begin{proof}
%We shall need to decompose the Laplacian into vertical and horizontal parts.
We have
\begin{equation*}
\left( \nabla^*\nabla \alpha \right)^v = \left( \nabla^*\nabla a \right)^v + \left( \nabla^*\nabla b \right)^v.
\end{equation*}
For the first term, we write
\begin{equation}\label{laplacev:avfirst}
\left( \nabla^*\nabla a \right)^v = \nabla^*\nabla^v a + \left( \nabla^* \nabla^h a \right)^v
\end{equation}
and
\begin{equation*}
\begin{split}
\left( \nabla^* \nabla^h a \right)^v & = \left( \nabla^* \left( \nabla^h \! f_i \zeta_i + f_i \nabla^h \zeta_i \right) \right)^v \\
& = \left( \nabla^* \nabla^h f_i \right) \zeta_i + \left( \nabla^h f_i \nabla^*\zeta_i - \nabla f_i \intprod \bar{\omega}_i \right)^v + f_i \nabla^* \nabla^h \zeta_i \\
& = \left( \nabla^* \nabla^h f_i \right) \zeta_i + 4 f_i \zeta_i
\end{split}
\end{equation*}
where we have used Lemma \ref{lemma:gradgammai}. Then (\ref{laplacev:avfirst}) yields
\begin{equation}\label{laplacev:av}
\left( \nabla^*\nabla a \right)^v = \nabla^*\nabla^v a + 4 a + \nabla^* \nabla^h f_i \zeta_i
\end{equation}
which gives the case $b = 0$ of the formula (\ref{laplacev:maineq}).

Assuming now that $Y$ is horizontal and $b$ is of type $(0,1),$ we compute
\begin{equation*}
\begin{split}
0 \equiv \nabla_Y \LA \zeta_i , b \RA & = \LA \nabla_Y \zeta_i , b \RA + \LA \zeta_i, \nabla_X b \RA \\
\LA \zeta_i, \nabla_Y b \RA & = - \LA \bar{\omega}_i(Y, -) , b \RA.
\end{split}
\end{equation*}
For $U$ vertical, we have
\begin{equation*}
\begin{split}
0 \equiv \nabla_U \LA \zeta_i , b \RA & = \LA \nabla_U \zeta_i , b \RA + \LA \zeta_i, \nabla_U b \RA
\end{split}
\end{equation*}
and
\begin{equation*}
\LA \zeta_i, \nabla_U b \RA = 0.
\end{equation*}
Hence, for $X \in TS^7,$ we have
\begin{equation*}
\left( \nabla_X b \right)^v = - \zeta_i \LA \bar{\omega}_i(X, -) , b \RA.
\end{equation*}
Next, let $\{e_j\}_{j = 1}^7$ be an orthonormal basis of vector fields that satisfies $\nabla_{e_j} e_k = 0$ at a given point. We compute
\begin{equation*}
\begin{split}
0 & \equiv \nabla_{e_j} \nabla_{e_j} \LA \zeta_i , b \RA \\
& = \LA \nabla_{e_j} \nabla_{e_j} \zeta_i , b \RA + 2 \LA \omega_i^\circ \left(e_j, - \right), \nabla_{e_j} b \RA + \LA \zeta_i, \nabla_{e_j} \nabla_{e_j} b \RA \\
& = - 6 \LA \zeta_i , b \RA + 2 \LA \omega_i^\circ \left(e_j, - \right), \nabla_{e_j} b \RA + \LA \zeta_i, \nabla_{e_j} \nabla_{e_j} b \RA.
\end{split}
\end{equation*}
Since $\LA \zeta_i, b \RA \equiv 0,$ we are left with
\begin{equation*}
\LA \zeta_i, \nabla_{e_j} \nabla_{e_j} b \RA = -2 \LA \omega_i^\circ \left(e_j, - \right), \nabla_{e_j} b \RA
\end{equation*}
%A simpler calculation shows that the Laplacian of a vertical 1-form is vertical. We conclude that
which may be rewritten as
\begin{equation}\label{laplacev:bv}
\left( \nabla^* \nabla b \right)^v = 2 \zeta_i \LA \bar{\omega}_i , d b \RA.
\end{equation}
%***This is equal to the i.p. on forms?
%\end{rmk}
Combining (\ref{laplacev:av}) and (\ref{laplacev:bv}) yields (\ref{laplacev:maineq}).
\end{proof}

\begin{rmk} The previous results carry over when $\nabla$ is coupled to a connection that is trivial along the fibers of the Hopf fibration.
\end{rmk}

\subsection{Vanishing of the vertical component} This subsection proves our vanishing theorem for the vertical component of an infinitesimal deformation in Coulomb gauge.

%We shall need the following Lemma.

\begin{lemma}\label{lemma:horizproperties} Let $B$ be an ASD instanton on $S^4,$ and put $A = \pi^*B.$ Assume that $f$ is a section of $\pi^* \gothg_E \to S^7$ satisfying
\begin{equation}\label{horizproperties:assumption}
\nabla_A^h f = 0.
\end{equation}
Then $f = \pi^*h$ is the pullback of a section of $\gothg_E \to S^4,$ with
\begin{equation}\label{horizproperties:conclusion}
\nabla_B h = 0.
\end{equation}
\end{lemma}
\begin{proof} We have
\begin{equation}\label{horizproperties:mixedpartials}
0 = \nabla^v \nabla^h f = \nabla^h \nabla^v f
\end{equation}
since $f$ is a $0$-form and the curvature $F_A$ is purely horizontal. By (\ref{horizproperties:assumption}), we have
$$D_A f = D_A^v f.$$
Applying $D_A$ to both sides, by Proposition \ref{lemma:dalpha}, we obtain
\begin{equation*}%\label{horizproperties:curvaturecomparison}
\begin{split}
D_A^2 f = \LB F_A , f \RB & = D_A D_A^v f \\
& = (D_A^v)^2 f + D_A^h D_A^v f + 2 D^v f (U_i) \bar{\omega}_i.
\end{split}
\end{equation*}
By (\ref{horizproperties:mixedpartials}) and the fact that $\left( D_A^v \right)^2 = 0$ for a pullback connection, the first two terms vanish, yielding
\begin{equation}\label{horizproperties:curvaturecomparison}
\LB F_A , f \RB = 2 \left( \nabla_{U_i} f \right) \bar{\omega}_i.
\end{equation}
Each side of (\ref{horizproperties:curvaturecomparison}) is of type $(0,2);$ however, the LHS is anti-self-dual and the RHS is self-dual. Therefore both sides of (\ref{horizproperties:curvaturecomparison}) must vanish, giving
$$\nabla^v f = 0.$$
Hence, $f$ is constant on the fibers, so $f = \pi^*h$ for a section $h$ of $\gothg_E.$ Then (\ref{horizproperties:assumption}) implies (\ref{horizproperties:conclusion}), as desired.
\end{proof}

\begin{thm}\label{thm:vertvanishing}\label{thm:vanishing} Let $B$ be an irreducible ASD instanton on a bundle $E \to S^4,$ and $A = \pi^* B$ %a $\G_2$-instanton on $S^7$ that is 
its pullback under the Hopf fibration. %For any infinitesimal deformation $\alpha \in \Omega^1 \left( \gothg_E \right)$ in Coulomb gauge ({\it i.e.} a solution of
If $\alpha \in \Omega^1 \left( \pi^*\gothg_E \right)$ satisfies $\scrL_A \alpha = 0,$ then the vertical part of $\alpha$ vanishes.
\end{thm}
\begin{proof} We write $\nabla = \nabla_A$ and $d = D_A$ throughout the proof. 

Decompose $\alpha = a + b$ into vertical and horizontal parts as above. Let $a = f_i \zeta_i,$ and write the self-dual part of $d^h b$ as 
\begin{equation*}
\left( d^h b \right)^+ = c_i \bar{\omega}_i .
\end{equation*}
%Then in a standard basis we have $c_1 = \frac{1}{2} !!!! \nabla_0 b_1 - \nabla_1 b_0 + \nabla_2 b_3 - \nabla_2 b_2,$ {\it etc.}
From Proposition \ref{lemma:dalpha} and (\ref{phistdsp2def}), our assumption $\scrL_A \alpha = 0$ implies
\begin{equation}\label{phidalpha}
\begin{split}
0 = \phi \rintprod d \alpha
& = \stackrel{(1,0)}{ \overbrace{ \quad \qquad \nu \rintprod d^v a \qquad \quad } } + \stackrel{(0,1)}{ \overbrace{ \phi_{1,2} \rintprod \left( d^h a + d^v b \right) }} \\
& \,\, + \phi_{1,2} \rintprod \left( 2 f_i \bar{\omega}_i + d^h b \right).
\end{split}
\end{equation}
Noting that $\bar{\omega}_i \rintprod \bar{\omega}_i = 2,$ according to(\ref{phistdsp2def}), the $(1,0)$-part of (\ref{phidalpha}) comes out to
\begin{equation*}
\begin{split}
0 & = \zeta_1 \left( *d^v a(U_1) + 4 f_1 + 2 c_1 \right)  \\
& + \zeta_2  \left( *d^v a(U_2) + 4 f_2 + 2 c_2 \right) \\
& + \zeta_3 \left( *d^v a(U_3) - 4 f_3 - 2 c_3 \right)
\end{split}
\end{equation*}
where $U_j$ is the dual vector field to $\zeta_j,$ for $j = 1,2,3,$ as above. We conclude that
\begin{equation}\label{ciformulafirst}
\begin{split}
c_1 & = -2 f_1 - \frac12 *d^v a (U_1) \\
c_2 & = -2 f_2 - \frac12 *d^v a (U_2) \\
c_3 & = -2 f_3 + \frac12 *d^v a (U_3).
\end{split}
\end{equation}
We rewrite this as
\begin{equation}\label{ciformula}
c_i \zeta_i = -2a + \frac{1}{2} \mu(a)
\end{equation}
where $\mu: \Omega^1_v \to \Omega^1_v,$ defined by (\ref{ciformulafirst}), obeys
\begin{equation}\label{munorm}
|\mu (a)| = |*d^v a | = |d^v a |.
\end{equation}
Now, the decomposition of the Laplacian (\ref{laplacev:maineq}) reads
\begin{equation}\label{laplacedecompfordef}
\begin{split}
\left( \nabla^* \nabla \alpha \right)^v & = \nabla^* \nabla^v a + 4 a + 4 c_i\zeta_i  + \left( \nabla^* \nabla^h \! f_i \right) \zeta_i \\
& = \nabla^* \nabla^v a - 4 a + 2 \mu(a)  + \left( \nabla^* \nabla^h \! f_i \right) \zeta_i 
\end{split}
\end{equation}
where we have used (\ref{ciformula}). Returning to (\ref{LAsquared}), we have
\begin{equation*}
\begin{split}
0 = \mathscr{L}_A^2 \alpha & = 2 \phi_{std} \rintprod \alpha + \nabla^*\nabla \alpha - 2 \LB F_A \rintprod \alpha \RB + 6 \alpha \\
& = \nabla^*\nabla \alpha - 2 \LB F_A \rintprod b \RB + 6 \alpha.
\end{split}
\end{equation*}
Taking the vertical part and inserting (\ref{laplacedecompfordef}), since $F_A$ is purely horizontal, we obtain
\begin{equation}\label{decomp0}
\begin{split}
0 = \left( \mathscr{L}_A^2 \alpha \right)^v = \nabla^* \nabla^v a + 2a + 2 \mu(a)  + \left( \nabla^* \nabla^h \! f_i \right) \zeta_i.
\end{split}
\end{equation}
Applying the Bochner formula on $S^3$ yields
\begin{equation}\label{decomp1}
0 = \left( d^* d^v + d^v d^* \right) a + 2 \mu(a)  + \left( \nabla^* \nabla^h \! f_i \right) \zeta_i.
\end{equation}
We now split $a$ into fiberwise closed and coclosed parts
$$a = a_{cl} + a_{cocl}$$
and write
$$a_{cocl} = g_i \zeta_i.$$
Taking an inner product with $a_{cocl}$ in (\ref{decomp1}), and integrating over $S^7,$ yields
\begin{equation}\label{decomp2}
0 = \| d^v a \|^2 + 2 (a_{cocl}, \mu(a))  + \left( g_i, \nabla^* \nabla^h \! g_i  \right).
\end{equation}
Here we have used the orthogonality between fiberwise closed and coclosed vertical 1-forms, the fact that $d^*a_{cocl} = (d^v)^*a_{cocl} = 0,$ and Lemma \ref{lemma:coclosed}.

Note that the first eigenvalue of the Hodge Laplacian on coclosed 1-forms on $S^3$ is 4, so
$$\| d^v a \| \geq 2 \| a_{cocl} \|.$$
Inserting this into (\ref{decomp2}), using Cauchy-Schwarz %#v3
and (\ref{munorm}), we obtain
\begin{equation*}
\begin{split}
0 &\geq 2 \| a_{cocl} \| \| d^v a \| - 2 \| a_{cocl} \| \| \mu(a) \|  + \sum_{i = 1}^3 \| \nabla^h \! g_i \|^2 \\
& \geq 2 \| a_{cocl} \| \| d^v a \| - 2 \| a_{cocl} \| \| d^v a \| +\sum_{i = 1}^3 \| \nabla^h \! g_i \|^2 \\
& \geq \sum_{i = 1}^3 \| \nabla^h \! g_i \|^2.
\end{split}
\end{equation*}
This yields
$$\nabla^h g_i \equiv 0$$
for $i = 1,2,3.$
Since $B$ is assumed irreducible, Lemma \ref{lemma:horizproperties} implies that $g_i \equiv 0,$ and hence $a_{cocl} \equiv 0.$

But then $a = a_{cl}$ is fiberwise closed, so $\mu(a) = 0$ by (\ref{munorm}). Since the RHS of (\ref{decomp0}) is then a positive operator, %#v3  % if $\mu(a) = 0,$
we conclude that $a \equiv 0,$ as desired.
\end{proof}

\subsection{Space of infinitesimal deformations} This subsection proves our main theorem, which calculates the deformations of a pulled-back instanton on $S^7$ in terms of the ASD deformations on $S^4.$

Consider the operator %#v3
$$\delta = \left( \phi \rintprod d(-) \right)^h : \Omega^1_h \to \Omega^1_h$$
given by the restriction of $\phi \rintprod d(-)$ to the space of horizontal 1-forms. For $b \in \Omega^1_h,$ we have 
\begin{equation*}%\label{phidesc}
\delta(b) = \left( \phi \rintprod d b \right)^h = \left( \phi^{(1,2)} \rintprod d^v b \right)^h.
\end{equation*}
%We restrict to the $S^3$ fiber inside $\{ (x,0) \mid x \in \H \}.$ Then $U^j = 
Let $S^3_0 \subset \{ (x, 0) \} \subset \H^2$ be the fiber of the Hopf fibration contained in the $x$-axis, and let
$$\delta_0= \left. \delta \right|_{S^3_0}.$$
Notice that
\begin{equation*}
\begin{split}
dy^j \rintprod \zeta_i & = \left( \zeta_i^y \right)^j = \omega_{ijk} y^k.
\end{split}
\end{equation*}
We may therefore define a local frame for $\Omega^1_h$ near $S^3_0$ by
\begin{equation*}
\bar{e}^j = dy^j - \zeta_i \omega_{ijk} y^k, \qquad j = 0, \ldots, 3.
\end{equation*}
We calculate
\begin{equation*}
\begin{split}
\left. d \bar{e}^j \right|_{S_0^3} & = \omega_{ijk} \zeta_i \wedge \bar{e}^k \\
& = \left. d^v \bar{e}^j \right|_{S_0^3}.
\end{split}
\end{equation*}
Then
\begin{equation*}
\begin{split}
\delta_0 \left( \bar{e}^j \right) = \left( \phi \rintprod d^v \bar{e}^j \right)^h & = - \left( \omega_{1jk} \omega_{1 k \ell} + \omega_{2jk} \omega_{2 k \ell} - \omega_{3jk} \omega_{3 k \ell}  \right) \bar{e}^\ell \\
\end{split}
\end{equation*}
and
\begin{equation}\label{phi0ebar}
\delta_0 \left( \bar{e}^j \right) = \bar{e}^j.
\end{equation}

\begin{lemma}\label{lemma:horiz} The kernel of $\delta_0$ consists of sections of the form
\begin{equation*}
L_{ij} x^i \bar{e}^j
\end{equation*}
where $L \in \mathfrak{F}$ is a Fueter map of $\R^4,$ according to (\ref{fueter}) above.
\end{lemma}
\begin{proof}
Let $\alpha = \alpha_j \bar{e}^j$ be an arbitrary section of $\Omega^1_h$ near $S^3_0.$ According to (\ref{phi0ebar}) and (\ref{phis30}), we have
\begin{equation}\label{phi0fueter}
\begin{split}
\delta_0 \left( \alpha \right) & = \phi \rintprod \left( U_i \left( \alpha_j \right) \zeta^i \wedge \bar{e}^j  \right) + \alpha_j \bar{e}^j \\
& = \left( U_1 \left( \alpha_j \right) \omega_{1 j \ell} + U_2 \left( \alpha_j \right) \omega_{2 j \ell} - U_3 \left( \alpha_j \right) \omega_{3 j \ell} + \alpha_\ell \right) \bar{e}^\ell \\
& =: \beta_\ell \bar{e}^\ell.
\end{split}
\end{equation}
Further, we calculate
\begin{equation*}
\begin{split}
\delta_0^2 \left( \alpha \right) \rintprod \bar{e}^i & = U_1 \left( \beta_\ell \right) \omega_{1 \ell i} + U_2 \left( \beta_\ell \right) \omega_{2 \ell i} - U_3 \left( \beta_\ell \right) \omega_{3 \ell i} + \beta_i \\
& = \omega_{1 \ell i} U_1 \left(U_1 \alpha_j \omega_{1j\ell} + U_2 \alpha_j \omega_{2j\ell} - U_3 \alpha_j \omega_{3j\ell} + \alpha_\ell \right) \\
& \quad + \omega_{2 \ell i}  U_2 \left( U_1 \alpha_j \omega_{1j\ell} + U_2 \alpha_j \omega_{2j\ell} - U_3 \alpha_j \omega_{3j\ell} + \alpha_\ell \right) \\
& \quad - \omega_{3 \ell i} U_3 \left( U_1 \alpha_j \omega_{1j\ell} + U_2 \alpha_j \omega_{2j\ell} - U_3 \alpha_j \omega_{3j\ell} + \alpha_\ell \right) \\
& \quad + \beta_i \\
& = - \left( U_1^2 + U_2^2 + U_3^2 \right) \alpha_i \\
& \quad + \omega_{2j\ell} \omega_{1 \ell i} U_1 U_2\left( \alpha_j \right) + \omega_{1 j \ell} \omega_{2 \ell i} U_2 U_1 \left( \alpha_j \right) \\
& \quad - \omega_{3j\ell} \omega_{1 \ell i} U_1 U_3 \left( \alpha_j \right) - \omega_{1 j \ell} \omega_{3 \ell i} U_3 U_1 \left( \alpha_j \right) \\
& \quad - \omega_{3j\ell} \omega_{2 \ell i} U_2 U_3 \left( \alpha_j \right) - \omega_{2 j \ell} \omega_{3 \ell i} U_3 U_2 \left( \alpha_j \right) \\
& \quad + \left( \beta_i - \alpha_i \right) + \beta_i \\
& = - \left( U_1^2 + U_2^2 + U_3^2 \right) \alpha_i  + \omega_{3 j i} \LB U_1, U_2 \RB \alpha_j + \omega_{2 j i} \LB U_1, U_3 \RB \alpha_j - \omega_{1 j i} \LB U_2, U_3 \RB \alpha_j  \\
& \quad + 2\beta_i - \alpha_i  \\
& = - \left( U_1^2 + U_2^2 + U_3^2 \right) \alpha_i  + 2 \left( - \omega_{3 ji } U_3 + \omega_{2ji} U_2 + \omega_{1ji} U_1 \right) \alpha_j \\
& \quad + 2\beta_i - \alpha_i.
\end{split}
\end{equation*}
Again applying (\ref{phi0fueter}), we obtain
\begin{equation}\label{horiz:phi02}
\delta_0^2 \left( \alpha \right) \rintprod \bar{e}^i = - \left( U_1^2 + U_2^2 + U_3^2 \right) \alpha_i  + 4\beta_i - 3 \alpha_i.
\end{equation}
Notice from (\ref{gradgammai:gradvh}) that $\nabla_{U_i} U_i = 0.$ 
\begin{comment} that
\begin{equation*}
\begin{split}
\nabla_{U_1} U_1 & = \pi_{S^3} \left( U_1 \left( x^j \right) \omega_{1\ell j} \frac{\partial }{\partial x^\ell} \right) \\
& = \pi_{S^3} \left( x^k \omega_{1kj} \omega_{1\ell j} \frac{\partial }{\partial x^\ell} \right) \\
& = \pi_{S^3} \left( x^\ell \frac{\partial }{\partial x^\ell} \right) = 0
\end{split}
\end{equation*}
and similarly for $U_2$ and $U_3.$ 
\end{comment}
Hence the first term on the RHS of (\ref{horiz:phi02}) is the Laplace-Beltrami operator on $S^3_0,$ applied component-wise to $\alpha.$
We conclude that
\begin{equation*}
\delta_0^2 \left( \alpha \right) = \Delta_{S^3} \alpha + 4 \delta_0 \left( \alpha \right) - 3 \alpha.
\end{equation*}
In particular, if $\delta_0(\alpha) = 0,$ we have
\begin{equation*}
\Delta_{S^3} \alpha_j = 3 \alpha_j.
\end{equation*}
Hence, $\alpha_j$ lies in the first nonzero eigenspace, and is the restriction of a linear function on $\R^4:$
$$\alpha_j = x^i L_{ij}.$$
Substituting back into (\ref{phi0fueter}), we have
\begin{equation*}
\delta_0(\alpha) = x^i \left( L_{ij} + \omega_{1ik} L_{k\ell}\omega_{1 \ell j} + \omega_{2 i k} L_{k\ell}\omega_{2 \ell j} - \omega_{3 i k} L_{k \ell}\omega_{3 \ell j} \right) \bar{e}^j = 0
\end{equation*}
which recovers (\ref{fueter}), as desired.
\end{proof}

\begin{prop}\label{prop:kerphi} The kernel of $\delta$ is the %$\pi^* C^\infty _{S^4}$-subsheaf
subspace of $\Omega^1_h$ given by
\begin{equation}\label{3subsheaves}
\pi^* \Omega^1_{S^4} \oplus  I_1 \left( \pi^* \Omega^1_{S^4} \right) \oplus I_2 \left( \pi^* \Omega^1_{S^4} \right).
\end{equation}
Here the complex structures $I_1$ and $I_2$ act pointwise on $\Omega^1_h \subset \Omega^1_{\R^8}.$
\end{prop}
\begin{proof} Over the fiber $S^3_0,$ this follows directly from Lemma \ref{lemma:horiz} and the description (\ref{fueterfactors}) of $\mathfrak{F}.$ Since $\Sp(2)$ preserves $\pi^* \Omega^1_{S^4}$ and commutes with $I_1$ and $I_2,$ the subspace (\ref{3subsheaves}) is also invariant under $\Sp(2).$ Hence, it agrees with $\ker \delta$ globally.
\end{proof}

\begin{thm}\label{thm:generaldefos} Let $A = \pi^*B$ be the pullback of an irreducible ASD instanton on $S^4,$ and let
$$V = \ker \left( D_B^+ \oplus D_B^* \right) \subset \Omega^1_{S^4} \left( \gothg_E \right)$$
denote the space of infinitesimal deformations of $B.$ The space of infinitesimal deformations of $A,$ as a $\G_2$-instanton, is given by
\begin{equation}\label{generaldefos:maineq}
\ker \scrL_A = \pi^* V \oplus  I_1 \left( \pi^* V \right) \oplus I_2 \left( \pi^* V \right) \subset \Omega^1_{S^7} \left( \gothg_{\pi^*E} \right).
\end{equation} %***need to think through Coulomb again?
For structure group $\SU(2),$ this has dimension $3 \left( 8\kappa - 3 \right).$
\end{thm}
\begin{proof} Let $\alpha \in \ker \scrL_A.$ According to Theorem \ref{thm:vertvanishing}, the vertical component of $\alpha$ vanishes, so it is a global section of $\Omega^1_h (\gothg_E).$ The image $\delta(\alpha) \in \Omega^1_h$ under the horizontal component of the deformation operator must therefore vanish; by Proposition \ref{prop:kerphi}, we have
$$\alpha = \alpha_{0} + \alpha_{1} + \alpha_{2}$$
according to (\ref{3subsheaves}). 
%Decompose $\alpha$ according to (\ref{3subsheaves}):

It remains to show that $\scrL_A(\alpha) = 0$ if and only if $d^*\alpha_{i} = 0$ and $(d^h\alpha_{i})^+ = 0,$ for $i = 0,1,2.$ Clearly $d^*\alpha = 0$ if and only if $d^*\alpha_{i} = 0 \, \forall \, i.$ %for $i = 0,1,2,$ since $d^*$ commutes with the complex structures $I_i.$
We have
\begin{equation}\label{dhalphaplus}
(d^h \alpha)^+ = \sum_{i = 1}^3 (d^h \alpha_{i} )^+
\end{equation}
where $(d^h \alpha_{i} )^+$ %$\in I_i \left( \pi^* \Omega^{2+}_{S^4} \right)$
are linearly independent. %, for $i = 0,1,2.$
But the map $\phi \, \rintprod $ is a bundle isomorphism from $\Omega^{2+}_h$ to $\Omega^1_v,$ so
$$\left(\phi \rintprod d \alpha \right)^v = \phi \rintprod (d^h \alpha)^+ = 0$$
if and only if $(d^h \alpha)^+ = 0.$ By (\ref{dhalphaplus}), this implies that $(d^h \alpha_{i} )^+ = 0$ for $i = 0,1,2,$ as desired.
%The component in each factor of (\ref{3subsheaves}) must therefore be mapped to zero individually, which can only be true if $\alpha$ lies in the space (\ref{generaldefos:maineq}).

The dimension formula for structure group $\SU(2)$ follows from the Atiyah-Hitchin-Singer Theorem \cite{ahs}. %(see also Appendix A).
\end{proof}

\section{Global picture}\label{sec:global}

In this section, we discuss the global structure of the components of the moduli space obtained by pullback under the Hopf fibration. We first prove Theorem \ref{thm:grassmann} on the structure of the $\kappa = 1$ component. For higher charge, the picture necessarily involves Hermitian-Yang-Mills connections on the twistor space $\C \P^3 \to S^4.$

\subsection{Proof of Theorem \ref{thm:grassmann}}
Let $W$ be given by (\ref{wdef}). Taking $A_0$ in the gauge (\ref{standardg2instanton}), we may define the smooth 5-dimensional family of connections
\begin{equation*}
V = \{\exp(w)^* A_0 \mid w \in W \} \subset \scrA_E.
\end{equation*}
%$W \simeq \SL_2 \left( \H \right) / \Sp(2) \simeq B^5$ denote .
By the construction of \S \ref{ss:hopf}, this family is equal to the pullback by the Hopf fibration of the 5-dimension family of unit-charge ASD instantons on $S^4.$
Further define a smooth map
\begin{equation}\label{firstg2map}
\begin{split}
\Spin(7) \times V & \to \scrA_E \\
\left(\sigma , A \right) & \mapsto \sigma^*A .
\end{split}
\end{equation}
By construction, the image of (\ref{firstg2map}) consists of $\G_2$-instantons.

Denote the principal bundle
$$Q = \Spin(7) \to \Spin(7) / \Sp(2) \times \U(1) .$$
%By the construction of \S \ref{ss:hopf},
The stabilizer of $V$ is $\Sp(2) \times \U(1) \subset \Spin(7),$ %#v3
which acts on $V,$ modulo gauge, by the 5-dimensional representation of $\Sp(2)$ and the trivial representation of $\U(1).$ Taking the quotient by the gauge group $\scrG_E,$ the map (\ref{firstg2map}) descends to a smooth map from the associated bundle
$$X = Q \times_{\Sp(2) \times \U(1)} V$$
to the space of connections modulo gauge:
\begin{equation*}%\label{secondg2map}
\Phi: X %= Q \times_{\Sp(2) \times \U(1)} V
\to \scrA_E / \scrG_E.
\end{equation*}
Notice that
$$\Spin(7)/ \Sp(2) \times \U(1) \cong \SO(7) / \SO(5) \times \SO(2) = \G^{or}(5,7).$$
Hence, $X$ is equal to the vector bundle over $\G^{or}(5,7)$ associated to the standard representation of $\SO(5),$ {\it i.e.}, the tautological 5-plane bundle. This is a 15-dimensional manifold.

We claim that the above map $\Phi$ is a proper embedding. %diffeomorphism onto its image. %This follows if the differential is injective, and the limit at the boundary of $M$ tends to infinity.
For each $A \in V,$ the image of the differential of $\Phi$ is equal (modulo gauge) to the space of linear deformations $\scrV_{A}$ given by (\ref{15deformations}).
By examining (\ref{stdg2instcurv}), it is easy to see that $F_{A_0}(x,y)$ spans $\Lie(\Sp(2))$ as $(x,y)$ varies over $S^7;$ the same is true of each element of the 5-dimensional family $V.$ %By Proposition \ref{prop:15deformations},
Hence, according to Proposition \ref{prop:15deformations}, $\dim(\scrV_A) = 15,$ and the differential of $\Phi$ has full rank at each point of $V.$ Since the construction is invariant under the action of $\Spin(7),$ the same is true at each point of $X.$ Hence, the map $\Phi$ is an embedding. % (in any Sobolev topology).

Next, note that as $x \in X$ tends to infinity in a fiber $V,$ the curvature of $\Phi(x)$ concentrates along an associative great sphere (the preimage under the Hopf fibration of a point in $S^4$). Hence, $\Phi(x)$ tends to infinity in $\scrA_E / \scrG_E.$ Since the base space $\G^{or}(5,7)$ is compact, this shows that $\Phi$ is a proper map.

Now, since for each $x \in X,$ $\Phi(x)$ is equivalent modulo $\Spin(7)$ to a pullback from $S^4,$ Theorem \ref{thm:generaldefos} implies that the space of infinitesimal deformations at $\Phi(x)$ has dimension 15. Since $\dim(X) = 15$ and $\Phi$ is a proper embedding,
we conclude that $\Phi$ is in fact a diffeomorphism onto a connected component of the $\G_2$-instanton moduli space.
\qed

\begin{rmk} By the same construction, we may compactify $X$ fiberwise to a $\bar{B}^5$-bundle, %$\bar{X},$
where a boundary point records bubbling along an associative great sphere. % together with the residual data of the fibration. %The 14-dimensional space $\partial \bar{X}$ is then a 2-sphere bundle over the space of associative great spheres, {\it i.e.}, the Cayley Grassmannian of $\R^8.$ %with fiber $\SU(2)_{r,\tilde{diag}} / \LA \exp \left( J \right) \RA.$
%We also note that the same construction can be made with instantons of higher charge, but the picture becomes more complicated; see the discussion in \S \ref{ss:discussion} below. %#v3 
\end{rmk}

%\subsection{The $\kappa = 1$ family}

\subsection{Chern-Simons functional and Hermitian-Yang-Mills connections}

Given two connections $A$ and $A_1$ on a bundle $E,$ let $a = A - A_1$ and define the relative Chern-Simons 3-form:
\begin{equation}
cs(A, A_1) = - Tr \left( a \wedge \left( F_{A_1} + \frac{1}{2} d_{A_1}a + \frac{1}{3} a \wedge a \right) \right).
\end{equation}
This satisfies
$$d cs(A, A_1) = Tr \left( F_A \wedge F_A - F_{A_1} \wedge F_{A_1} \right).$$
On a 7-manifold $M$ with $\G_2$-structure, we may define the global relative \emph{Chern-Simons functional}
\begin{equation}
CS_{\psi}(A, A_1) = \frac{-1}{4 \pi^4}\int_M \psi \wedge cs(A, A_1).
\end{equation}
%We have the following simple generalization of Lemma \ref{lemma:instapullback}.
This normalization is chosen so that the formula in the following proposition will be integer-valued. %#v3
Let
\begin{equation}\label{bigpi}
\Pi : S_{std}^7 \to \C \P^3
\end{equation}
be the natural projection for the standard complex structure, $I_1.$

\begin{prop}\label{prop:hym} Let $B$ be a connection on an $\SU(n)$-bundle $E \to \C\P^3.$ Then $A = \Pi^*B$ is a $\G_2$-instanton on $S^7_{std}$ if and only if $B$ is Hermitian-Yang-Mills.

Given any two such connections $(E,B)$ and $(E_1,B_1)$ for which $\Pi^* E \cong \Pi^*E_1,$ we have %***
\begin{equation*}
CS_\psi \left(A, A_1 \right) = \LA \LB \omega_{FS} \RB \cup \left( c_2(E) - c_2(E_1) \right), \LB \C\P^3 \RB \RA.
\end{equation*}
%The same holds with $\tilde{\C \P^3}$ and $S^7_{sq}$?
\end{prop}
\begin{proof} By $\SU(4)$-invariance, it suffices to consider the point $p = (1, 0, \ldots, 0) \in S^7.$ We shall write
\begin{equation}
T_p S^7 = \R_{x_1} \oplus \C^3.
\end{equation}
Then, from the expression (\ref{psi0kahler}), we have
\begin{equation}
\phi_{std}(p) = \frac{\partial}{\partial x_0} \intprod \Psi_0 = dx^1 \wedge \omega + \Re dz_2 \wedge dz_3 \wedge dz_4
\end{equation}
where $\omega$ is the standard K{\"a}hler form on $\C^3,$ given by (\ref{stdkahlerforms}). Letting
$$F = F_A(p) = \left. \Pi^*F_B \right|_{T_p S^7} = \left. F_B \right|_{\C^3}$$
we have
\begin{equation}
\phi_0 \rintprod F = dx^1 \left( \omega . F \right) + \Re \, dz_2 \wedge dz_3 \wedge dz_4 \rintprod F^{2,0} .
\end{equation}
Therefore, $A$ is a $\G_2$-instanton if and only if
$$\omega .F = 0, \quad F^{2,0} = 0$$
which is to say, $B$ is Hermitian-Yang-Mills on $\C \P^3.$

Next, we have
\begin{equation}\label{thingtointegrate}
\begin{split}
 - 4 \pi^4 CS_\psi (A, A_1) & = \int_{S^7} \psi \wedge cs(A, A_1) \\
& = \frac{1}{4} \int_{S^7} d \phi \wedge cs(A, A_1) \\
& = - \frac{1}{4} \int_{S^7} \phi \wedge d cs(A, A_1) \\
& = - \frac{1}{4} \int_{S^7} \phi \wedge Tr \left( F_A \wedge F_A  - F_{A_1} \wedge F_{A_1} \right).
\end{split}
\end{equation}
Assume that both $A$ and $A_1$ are pullbacks of Hermitian-Yang-Mills connections on $\C \P^3.$ Computing at $p$ as above, we have $F = F_A^{1,1}$ and are left with
\begin{equation}\label{thisexpression}
\phi_0 \wedge Tr \left( F \wedge F \right) = dx^1 \wedge \omega \wedge Tr \left( F \wedge F \right).
\end{equation}
Since the Fubini-Study form on $\C\P^3$ satisfies
$$\Pi^*\omega_{FS} = \frac{\omega}{\pi}$$
on $S^7,$ the expression (\ref{thisexpression}) agrees globally with
$$\pi \zeta_1 \wedge \Pi^* \left( \omega_{FS} \wedge Tr \left( F_B \wedge F_B \right) \right).$$
Integrating (\ref{thingtointegrate}) over the fibers of $\Pi,$ we obtain %***think about this slightly
\begin{equation*}
\begin{split}
-4 \pi^4 CS_\psi (A, A_1) & = - \frac{\pi}{4} \int_{S^7} \zeta_1 \wedge \Pi^* \left( \omega_{FS} \wedge Tr \left( F_B \wedge F_B - F_{B_1} \wedge F_{B_1} \right) \right) \\
& = - \frac{\pi^2}{2} \int_{\C\P^3} \omega_{FS} \wedge Tr \left( F_B \wedge F_B - F_{B_1} \wedge F_{B_1} \right) \\
& = - \frac{\pi^2}{2} \LA \LB \omega_{FS} \RB \cup 8 \pi^2 \left(c_2(E) - c_2(E_1) \right) , \LB \C \P^3 \RB \RA \\
& = -4 \pi^4 \LA \LB \omega_{FS} \RB \cup \left(c_2(E) - c_2(E_1) \right) , \LB \C \P^3 \RB \RA
%& = - \frac{\pi^2}{2} \int_{\C\P^3} \omega_FS \wedge Tr \left( F_B \wedge F_B - F_{B_1} \wedge F_{B_1} \right)
%& = const . \left( ch_2(E_1) - ch_2(E) \right)
\end{split}
\end{equation*}
as claimed.
\end{proof}

%\begin{rmk}\label{rmk:lastrmk}
\subsection{Discussion}\label{ss:discussion}
Theorem \ref{thm:generaldefos} can be explained in light of Proposition \ref{prop:hym}, as follows.
Owing to the Ward correspondence, the space of infinitesimal deformations of the pullback of an ASD instanton to the twistor space $\C \P^3 \to S^4,$ as a Hermitian-Yang-Mills connection, is the complexification of the space of ASD deformations. %These deformations are unobstructed, since the same is true of the space of ASD deformations.
%But the fibration (\ref{bigpi}) is not unique; 
In fact, there is an $S^1$ family of complex structures
\begin{equation*}%\label{cp3theta}
\cos (\theta) I_1 + \sin (\theta) I_2 %\mid 0 \leq \theta < 2 \pi \right\}
\end{equation*}
such that the quotient map $\Pi_\theta$ %$S^7 \to \C \P^3_\theta$
satisfies Proposition \ref{prop:hym} and factorizes the given Hopf fibration:
\[
\xymatrix{
S^7 \ar@{->}[r]^{\Pi_\theta} \ar@{->}[dr]_{\pi} & \C\P^3 \ar@{->}[d] \\
 & S^4.}
\]
The span of the pullbacks by $\Pi_\theta$ of the deformation space over $\C \P^3,$ for $\theta \in S^1,$ gives the larger space (\ref{generaldefos:maineq}).

%\begin{comment}
According to this argument, the fibration to $\C\P^3$ %and its conjugates
accounts for all of the infinitesimal deformations identified by Theorem \ref{thm:generaldefos}. %in (\ref{generaldefos:maineq}). %the deformations on $\C\P^3.$
As such, we expect that only the deformations coming from (\ref{bigpi}), or its 6-dimensional family of $\Spin(7)$ conjugates, %It's S^6!
are integrable. So for $\kappa > 1,$ the generic dimension of the pulled-back component of the $\G_2$-instanton moduli space should be\footnote{The $\kappa = 1$ instantons all fail to be ``generic'' due to the extra $\U(1)$ stabilizer generated by $\left( \begin{array}{cc}
I_3  & 0 \\
0 & I_3  \end{array} \right).$} %, which belongs to $\Spin(7)$ but not $\SU(4).$ %#v3
\begin{equation}\label{higherchargedimformula}
2\left( 8 \kappa - 3 \right) + 6 = 16 \kappa.
\end{equation}
%In view of (\ref{cp3theta}), 
%The $\kappa > 1$
These components appear to be singular along the subset of instantons coming from $S^4$ ({\it i.e.}, instanton bundles on $\C\P^3$ satisfying a reality condition), % (see Atiyah \cite{atiyahgeometry}),
which are themselves manifolds of dimension
$$ 8 \kappa - 3 + 10 = 8 \kappa + 7.$$
%per the proof of Theorem \ref{thm:grassmann}.
%It should be possible to verify these predictions in further work, and to determine the topology of these components to the extent possible. %would be interesting to prove that the components are smooth and of  (\ref{higherchargedimformula}) away from 
%\end{comment}

More generally, %#v3
we would like to know where the instantons obtained via pullback %fit in 
fit within the full moduli space of $\G_2$-instantons on $S^7.$ While %$\G_2$-instantons on $S^7$
examples that are unrelated to any fibration
may exist, %in light of the present work and \cite{yuanqicypullback},
the following guess is appropriate based on \cite{yuanqicyproduct} and the present work. % that these connected components are unique with their given Chern-Simons value.
For the statement, fix a reference Hermitian-Yang-Mills connection $B_1$ on an $\SU(n)$-bundle $E_1 \to \C \P^3,$ and let $E = \Pi^* E_1 \to S^7$ and $A_1 = \Pi^* B_1.$
%\end{rmk}

\begin{conj}[Donaldson]\label{conj:donaldson} Let $A$ be a $\G_2$-instanton on $E \to S^7_{std},$ for which
$$CS_{\psi}(A, A_1) \in \Z.$$
Then $A$ is equivalent, modulo the action of $\Spin(7)$ and $\scrG_E,$ to the pullback of a Hermitian-Yang-Mills connection on $\C \P^3.$
\end{conj}

%\end{rmk}

\appendix

\section{Squashed 7-sphere}\label{appendix:squashed}

%Using the notation of \S \ref{sec:generaldefos}, we give a proof that the squashed 7-sphere of Example \ref{ex:squashed} is nearly parallel.

As a check on our conventions, we include the following short proof that $\phi_{sq}$ is a nearly parallel $\G_2$-structure.

Let $\zeta_i$ and $\bar{\omega}_i$ be the $\Sp(2)$-invariant frames for $\Omega^1_v$ and $\Omega^{2+}_h$ given by Definition \ref{defn:coframes}, and $\nu$ the $(3,0)$ volume form of the Hopf fibration. %Also let $\bar{\nu} = *_{std} \nu$ be the $(0,4)$ volume form for $\Omega^1_h$ with respect to the standard metric.
The squashed $\G_2$-structure, defined in Example \ref{ex:squashed}, may be rewritten
\begin{equation*}
\phi_{sq} = \frac{27}{25} \left( \frac{1}{5} \nu - \zeta_1 \wedge \bar{\omega}_1 - \zeta_2 \wedge \bar{\omega}_2 - \zeta_3 \wedge \bar{\omega}_3 \right). 
\end{equation*}
More generally, for $a, b > 0,$ define the $\Sp(2) \Sp(1)$-invariant $\G_2$-structure
\begin{equation*}
\phi_{a,b} = a \nu - b \zeta_i \wedge \bar{\omega}_i.
\end{equation*}
One can check from (\ref{g2metric}) that the metric and volume form defined by $\phi_{a,b}$ are
\begin{equation*}
g_{a,b} = a^{-1/3} \left( a \left. g_{std} \right|_{T_v} +  b \left. g_{std} \right|_{T_h} \right), \qquad Vol_{a,b} = a^{1/3} b^2 Vol_{std}.
\end{equation*}
In this metric, we have %See 10/17/19
\begin{equation*}
\begin{split}
\left| \nu \right|_{g_{a,b}} & = a^{-1}, \qquad \qquad \left| \zeta_i \wedge \bar{\omega}_j \right|_{g_{a,b}} = \sqrt{2} b^{-1} \\
*_{g_{a,b}} \nu & = a^{-5/3} b^2 \bar{\nu}, \qquad *_{g_{a,b}} \left( \zeta_i \wedge \bar{\omega}_j \right) = \frac{ a^{1/3} }{2} \epsilon_{ik\ell} \zeta_k \wedge \zeta_\ell \wedge \bar{\omega}_j .
\end{split}
\end{equation*}
The dual 4-form is given by
\begin{equation*}%\label{psiab}
\begin{split}
\psi_{a,b} & = *_{g_{a,b}} \phi_{a,b} = a^{-2/3} b \left( b \bar{\nu} - a \epsilon_{ijk} \zeta_i \wedge \zeta_j \wedge \bar{\omega}_k \right).
\end{split}
\end{equation*}
In this notation, we have %The squashed $\G_2$-structure is
\begin{equation*}
\begin{split}
\phi_{sq} & = \phi_{\frac{27}{125},\frac{27}{25}} = \left( \frac{3}{5} \right)^3 \phi_{1,5}.
\end{split}
\end{equation*}
The associated metric and volume form are
\begin{equation*}
g_{sq} = \frac{9}{5} \left( \frac{1}{5} \left. g_{std} \right|_{T_v} +  \left. g_{std} \right|_{T_h} \right) , \qquad Vol_{sq} = \frac{3^7}{5^5} Vol_{std}.
\end{equation*}
To see that $\phi_{sq}$ is nearly parallel, we calculate as follows. According to (\ref{dgammai}), we have
\begin{equation*}%\label{dgammai}
d \zeta_i  = \epsilon_{ijk} \zeta_j \wedge \zeta_k + 2 \bar{\omega}_i.
\end{equation*}
Applying $d$ to (\ref{dgammai}), we obtain
$$0 = 2 \epsilon_{ijk} d \zeta_j \wedge \zeta_k + 2 d \bar{\omega}_i$$
and
\begin{equation*}
\begin{split}
d \bar{\omega}_i & = \epsilon_{ijk} \zeta_j \wedge d \zeta_k \\
& = 2 \epsilon_{ijk} \zeta_j \wedge \bar{\omega}_k + \epsilon_{ijk} \epsilon_{k\ell m} \zeta_j \wedge \zeta_{\ell} \wedge \zeta_m \\
& = 2 \epsilon_{ijk} \zeta_j \wedge \bar{\omega}_k.
\end{split}
\end{equation*}
Next, we compute
\begin{equation*}
\begin{split}
d \nu & = d \zeta_1 \wedge \zeta_2 \wedge \zeta_3 - \zeta_1 \wedge d \zeta_2 \wedge \zeta_3 + \zeta_1 \wedge \zeta_2 \wedge d \zeta_3 \\
& = 2 \left( \bar{\omega}_1 \wedge \zeta_2 \wedge \zeta_3 -  \zeta_1 \wedge \bar{\omega}_2 \wedge \zeta_3 +  \zeta_1 \wedge \zeta_2 \wedge \bar{\omega}_3 \right) \\
& = \epsilon_{ijk} \bar{\omega}_i \wedge \zeta_j \wedge \zeta_k.
\end{split}
\end{equation*}
\begin{comment}
In summary, we have
\begin{equation*}%\label{domeganui}
\begin{split}
d \bar{\omega}_i & = 2 \epsilon_{ijk} \zeta_j \wedge \bar{\omega}_k \\
d \nu & = \epsilon_{ijk} \bar{\omega}_i \wedge \zeta_j \wedge \zeta_k.
\end{split}
\end{equation*}
\end{comment}
We then have
\begin{equation*}
\begin{split}
d \phi_{a,b} & = a \epsilon_{ijk} \bar{\omega}_i \wedge \zeta_j \wedge \zeta_k - b \left( 2 \bar{\omega}_i + \epsilon_{ijk} \zeta_j \wedge \zeta_k \right) \wedge \bar{\omega}_i  + 2b \zeta_i \wedge \epsilon_{ijk} \zeta_j \wedge \bar{\omega}_k \\
& = \left( a + b \right) \epsilon_{ijk} \zeta_i \wedge \zeta_j \wedge \bar{\omega}_k - 2b \bar{\omega}_i \wedge \bar{\omega}_i \\
& = \left( a + b \right) \epsilon_{ijk} \zeta_i \wedge \zeta_j \wedge \bar{\omega}_k - 12b \bar{\nu}.
\end{split}
\end{equation*}
%From (\ref{psiab}), 
This gives
\begin{equation*}
d \phi_{1,5} = - \frac{12}{5} \psi_{1,5}.
\end{equation*}
After normalizing by $ \frac{12}{5\cdot 4} = \frac{3}{5},$ we have
\begin{equation*}
d \phi_{sq} = -4 \psi_{sq}
\end{equation*}
as claimed.

\begin{rmk} Notice that although $\phi_{1,1}$ and $\phi_{std}$ both correspond to the standard metric on $S^7,$ they are not isomorphic, since the former is not a nearly parallel $\G_2$-structure. In fact, Friedrich \cite{friedrichuniqueness} has shown that $\phi_{std}$ %with stabilizer $\Spin(7) \not \supset \Sp(2) \Sp(1),$
is the unique nearly parallel $\G_2$-structure (up to rotations) compatible with the round metric on $S^7.$ 
\end{rmk}

\bibliographystyle{amsinitial}
\bibliography{biblio}

\end{document}